\documentclass[11pt]{article}
\usepackage{layout}
\usepackage[makeroom]{cancel}
\usepackage{bbm}
\usepackage[affil-it]{authblk}

\usepackage{mathtools}

\mathtoolsset{showonlyrefs}

\usepackage{amsfonts,dsfont}
\usepackage{amsmath}
\usepackage{amssymb,bm}
\usepackage{amsthm}
\usepackage{graphicx} \usepackage{enumerate} \usepackage{multicol}
\usepackage{mathrsfs} \usepackage[all,cmtip]{xy}
\usepackage{enumerate,enumitem}
\usepackage[rgb,dvipsnames]{xcolor}
\usepackage{float}
\usepackage{ulem}
\usepackage{circuitikz}

\usepackage[colorlinks,linkcolor=blue,citecolor=blue,pagebackref,hypertexnames=false, breaklinks]{hyperref}
\usepackage{cite}

\usepackage{float}

\allowdisplaybreaks[4]

\newlength{\bibitemsep}
\newlength{\bibparskip}
\let\oldthebibliography\thebibliography
\renewcommand\thebibliography[1]{%
 \oldthebibliography{#1}%
 \setlength{\parskip}{\bibitemsep}%
 \setlength{\itemsep}{\bibparskip}%
}

\newtheorem{theorem}{Theorem}[section]

\newtheorem{assumption}[theorem]{Assumption}
\newtheorem{definition}[theorem]{Definition}
\newtheorem{proposition}[theorem]{Proposition}

\newtheorem{corollary}[theorem]{Corollary}
\newtheorem{lemma}[theorem]{Lemma}
\newtheorem{remark}[theorem]{Remark}
\newtheorem{example}[theorem]{Example}
\newtheorem{examples}[theorem]{Examples}
\newtheorem{foo}[theorem]{Remarks}

\newtheorem{conjecture}[theorem]{Conjecture}
\newtheorem{open question}[theorem]{Open Question}
\newtheorem{c/p}[theorem]{Conjecture/Proposition}

\newcommand{\brak}[1]{\left(#1\right)} 

\def\vint{\mathop{\mathchoice%
 {\setbox0\hbox{$\displaystyle\intop$}\kern 0.22\wd0%
 \vcenter{\hrule width 0.6\wd0}\kern -0.82\wd0}%
 {\setbox0\hbox{$\textstyle\intop$}\kern 0.2\wd0%
 \vcenter{\hrule width 0.6\wd0}\kern -0.8\wd0}%
 {\setbox0\hbox{$\scriptstyle\intop$}\kern 0.2\wd0%
 \vcenter{\hrule width 0.6\wd0}\kern -0.8\wd0}%
 {\setbox0\hbox{$\scriptscriptstyle\intop$}\kern 0.2\wd0%
 \vcenter{\hrule width 0.6\wd0}\kern -0.8\wd0}}%
 \mathopen{}\int}

\newcommand{\DF}{\mathcal{E}}

\newcommand{\B}{\mathbf B}

\newcommand{\ve}{\varepsilon}

\usepackage[textwidth=3.2cm]{todonotes}

\title{Besov class via heat semigroup on Dirichlet spaces III: BV
 functions and sub-{G}aussian heat kernel estimates. }

\author{Patricia Alonso-Ruiz, Fabrice Baudoin, Li Chen, Luke Rogers, Nageswari Shanmugalingam, Alexander Teplyaev}

\date{\today}

\begin{document}

\maketitle

\begin{abstract}

With a view toward fractal spaces, by using a Korevaar-Schoen space approach, we introduce the class of bounded variation (BV) functions in the general framework of strongly local Dirichlet spaces with a heat kernel satisfying sub-Gaussian estimates. Under a weak Bakry-\'Emery curvature type condition, which is new in this setting, this BV class is identified with a heat semigroup based Besov class. As a consequence of this identification, properties of BV functions and associated BV measures are studied in detail. In particular, we prove co-area formulas, global $L^1$ Sobolev embeddings and isoperimetric inequalities. It is shown that for nested fractals or their direct products the BV class we define is dense in $L^1$. The examples of the unbounded Vicsek set, unbounded Sierpinski gasket and unbounded Sierpinski carpet are discussed.

\end{abstract}

\tableofcontents

\section{Introduction}


In this paper we introduce and study functions of bounded variation 
on strongly local Dirichlet spaces 
which \textit{may not have Gaussian heat kernel bounds, but have sub-Gaussian heat kernel bounds} 
as given in~\eqref{eq:subGauss-upper} and satisfy a weak Bakry-\'Emery curvature condition~\eqref{WBECDintro}. 
We note that some properties that are usually taken for granted may not hold true in this setting because energy measures are not 
necessarily absolutely continuous with respect to a fixed measure
$\mu$. Therefore, unlike our analysis in \cite{ABCRST2} we can not develop and use locally Lipschitz functions, and need to develop a different set of tools.  

This introduction is devoted to giving an overview of the content of the paper and to providing a summary of the main results obtained.
The precise description of the sub-Gaussian heat kernel bounds and the definition of the Besov spaces that are investigated are presented in Section~\ref{S:Preliminaries}. Section~\ref{Section wBE} deals with a weak Bakry-\'Emery type curvature condition that is key in studying the notion of BV class introduced in Section~\ref{section BV}. Eventually, Section~\ref{S:Examples} presents several examples of spaces where the theory developed in previous sections applies.

\subsection*{An approach to BV functions in metric measure spaces}

Our approach to a theory of functions of bounded variation (BV)  is based on the study of  the $L^1$ Korevaar-Schoen class at the critical exponent. To explain our motivation and results, let us present this approach in the  general context of metric measure spaces.

Let $(X,d,\mu)$ be a locally compact complete metric measure space where $\mu$ is a Radon measure. For $\lambda >0$ and $p \ge 1$, we define the space $KS^{\lambda,p}(X)$ as the collection of all functions $f\in L^p(X,\mu)$ for which 
\[
\Vert f\Vert_{KS^{\lambda,p}(X)}^p:=\limsup_{r\to 0^+}\int_X\int_{B(x,r)}\frac{|f(y)-f(x)|^p}{r^{\lambda p} \mu(B(x,r))}\, d\mu(y)\, d\mu(x)<+\infty.
\]
The $L^p$--Korevaar-Schoen critical exponent of the space is then defined as
\begin{equation*}\label{e:alpha-critical}
\lambda_p^\#:=\sup \{ \lambda >0\, :\, KS^{\lambda,p}(X) \text{ contains non-constant functions}\}.
\end{equation*}
In the context of a complete metric measure space $(X,d,\mu)$  supporting a $1$-Poincar\'e inequality and where 
$\mu$ is doubling, one has $\lambda_p^\#=1$ for every $p \ge 1$. Note that, at the critical exponent $\lambda_2^\#=1$,  one can construct a Dirichlet form 
\[
\mathcal{E}(f,f) \simeq \Vert f\Vert_{KS^{1,2}(X)}^2
\]
with domain $KS^{1,2}(X)$ by using a choice of a Cheeger differential structure as in~\cite{Chee99}. This Dirichlet form is then strictly local and the intrinsic distance  $d_\mathcal{E}$ associated to $\mathcal{E}$ is bi-Lipschitz equivalent to the original metric $d$. We refer to \cite{MMS} and the references therein for further details. In that same framework, at the critical exponent $\lambda_1^\#=1$, one has $KS^{1,1}(X)=BV(X)$ and 
 \[
 \mathbf{Var} (f) \simeq \Vert f\Vert_{KS^{1,1}(X)}.
 \]

By contrast, in the context of the present paper, $(X,d,\mu)$ is a complete metric measure space for which $\mu$ is doubling (even Ahlfors regular) but for which $\lambda_2^\#=\frac{d_W}{2}$, where $d_W \ge 2$ is a parameter called the walk dimension of the space, see \cite{Gri,GrigLiu15}. At the critical exponent $\lambda_2^\#=\frac{d_W}{2}$ one has a strongly local (but not strictly local) Dirichlet form 
\[
\mathcal{E}(f,f) \simeq \Vert f\Vert_{KS^{\frac{d_W}{2},2}(X)}^2
\]
 with domain $KS^{\frac{d_W}{2},2}(X)$ whose heat kernel satisfies the sub-Gaussian estimates \eqref{eq:subGauss-upper}; see \cite[Corollary 3.4]{HZ} and 
Section~\ref{section BV},  
Remark \ref{Remark KS Dirichlet}. 
By analogy with the previous case, it seems then natural to study the corresponding $L^1$ critical exponent $\lambda_1^\#$ and the associated class $KS^{\lambda_1^\#,1}(X)$.  For the spaces we are interested in, which are primarily fractals or products of fractals, we shall see that this class $KS^{\lambda_1^\#,1}(X)$ has many of the expected properties of a BV class: it has co-area formulas
(see Theorem~\ref{coarea formula 1a}), existence of BV measures (see Theorem~\ref{equivalence BV}), and Sobolev embeddings
(see Theorem~\ref{Sobolev global}). The key assumption that yields these properties is the weak Bakry-\'Emery condition~\eqref{WBECDintro}.
 
For nested fractals~\cite{Sab99,Lind,KigB,BP,Ba98}, we prove that $\lambda_1^\#=d_H$ is the Hausdorff dimension of the space and that $KS^{\lambda_1^\#,1}(X)$ is dense in $L^1(X,\mu)$, see Theorem \ref{T:nested}. For the Sierpinski carpet~\cite{BP,Ba98,BB89,BB99} we prove that 
\begin{equation}\label{e:alpha-SC intro}
\lambda_{1}^\#\geqslant d_H-d_{tH}+1
\end{equation}
and conjecture that in fact there is an equality  in \eqref{e:alpha-SC intro}. Here, $d_{tH}$ is the topological-Hausdorff dimension defined in \cite{BBE15}.  
 
A key point in the study of the  $L^1$ Korevaar-Schoen classes $KS^{\lambda,1}(X)$ is Proposition \ref{KS-BEsov} which allows us to identify $KS^{\lambda,1}(X)$ with the heat semigroup based Besov class $\B^{1,\alpha}(X)$, $\alpha =\frac{\lambda}{d_W}$, that was introduced and extensively studied in our previous papers \cite{ABCRST1, ABCRST2}. In particular, we note that $\alpha_1^\#:=\frac{\lambda_1^\#}{d_W}$ is the critical parameter in the Besov scale of the classes $\B^{1,\alpha}(X)$, that is, for $\alpha$ 
larger than this threshold, the corresponding Besov classes only contain constant functions.  
Working directly with the Besov classes $\B^{1,\alpha}(X)$ has the advantage of setting us in the framework of~\cite{ABCRST1}, which allows to use a wide range of heat semigroup techniques paralleling  the methods developed in \cite{ABCRST2}. For this reason, most of our results are written for the Besov class $\B^{1,\alpha}(X)$ and the corresponding critical exponent $\alpha_1^\#$ rather than in terms of  $KS^{\lambda,1}(X)$ and $\lambda_1^\#$.
 
 \subsection*{Weak Bakry-\'Emery nonnegative curvature condition}
The main tool in this paper is the heat semigroup. In the Euclidean case, the deep connection between regularizing properties of the heat semigroup and  the theory of BV functions and sets of finite perimeter was uncovered by E. De Giorgi in the celebrated paper \cite{DeGiorgi54}. Among many other works, this connection was further developed and investigated by M. Ledoux in  \cite{Ledoux} (see also the references therein). The Bakry-\'Emery calculus shows that regularizing properties of the heat semigroup are intimately connected with Ricci curvature-type lower bounds on the underlying space, see \cite{BGL}. Thus, it should come as no surprise that the approaches of De Giorgi and Ledoux and the notions of isoperimetric inequalities and BV functions may be generalized to large classes of spaces for which Ricci curvature type lower bounds are well understood, like the now-extensively studied RCD$(0,\infty)$ spaces (see \cite{BGS14,CaMo18}) or sub-Riemannian spaces (see \cite{BK14,BB2}). In the context of the present paper, although the Bakry-\'Emery calculus is not available, the weak Bakry-\'Emery curvature condition introduced in \cite{ABCRST2} has a natural H\"older analogue which is the key assumption of our work. We shall say that the weak Bakry-\'Emery non-negative curvature condition $wBE(\kappa)$ is satisfied if there exist a constant $C>0$ and a parameter $0 < \kappa < d_W$ such that for every $t >0$, $g \in L^\infty(X,\mu)$ and  $x,y \in X$,
\begin{align}\label{WBECDintro}
| P_t g (x)-P_tg(y)| \le C \frac{d(x,y)^\kappa}{t^{\kappa/d_W}} \| g \|_{ L^\infty(X,\mu)}.
\end{align}

We prove in Theorem \ref{thm-xBE-FD} that fractional metric spaces for which $d_W>d_H\geq1$ 
must satisfy $wBE(\kappa)$ with $\kappa =d_W-d_H$.
Note that the hypothesis of Theorem~\ref{thm-xBE-FD} is  
stable under rough isometries in the sense of Barlow-Bass-Kumagai  \cite{B-JEMS,BBK06,BB04} and it is part of the general theory of fractional diffusions, see 
\cite{BP,Ba03,Ba98,Gri,KigB,Kig:RFQS,BaASC,BB89,BB99,BBKT} and references therein. 
In particular for nested fractals, Theorem~\ref{thm-xBE-FD} yields that $wBE(\kappa)$ is satisfied with $\kappa =d_W-d_H$ and in that case 
the value $d_W-d_H$ is optimal in the sense that $wBE(\kappa)$ is not satisfied for $\kappa >d_W-d_H$. Theorem~\ref{thm-xBE-FD} also proves that the Sierpinski carpet satisfies $wBE(\kappa)$ 
with  $\kappa= d_W-d_H $, however we conjecture  that in fact  the Sierpinski carpet satisfies  $wBE(\kappa)$ with  
$\kappa >d_W-d_H$.
 It will be a subject of future work to investigate whether $\kappa= d_W-d_H+d_{tH}-1$.
 
We would like to briefly comment on the \textit{curvature} interpretation of the weak Bakry-\'Emery condition $wBE(\kappa)$.  Our Theorem~\ref{thm-xBE-FD} holds only when $d_W>d_H\geq1$, which means that this theorem is applicable for low 
dimensional spaces. In such a low dimensional situation, geometrically speaking, there is no curvature. For nested fractals, the topological  and topological Hausdorff dimension are both 1. Thus, in some sense, they are analogues of lines which have zero curvature.  From a different perspective this corresponds to the Hodge-type theorems 
in~\cite{HT-Hodge,HT-curl} and Liouville-type theorems~\cite{gong2018li,hua2017liouville} 
and~\cite[Introduction and Section 4]{St99}.  The curvature interpretation of $wBE(\kappa)$ will only manifest itself in higher dimension when $d_H > d_W$, and this is why Subsection~\ref{subsec-products} is important, as it 
allows us to construct higher dimensional examples satisfying $wBE(\kappa)$. 
When $d_H>d_W$ one can expect the boundary of sets of finite perimeter to have a real geometry more complicated than that of Cantor sets. 

Beyond direct products of fractional spaces \cite{IRS13,St09p,St07p,St05p}, which in a sense are still flat, one could try to construct ``fractional manifolds, or fractafolds \cite{Str,ST}, with non-negative curvature" which would be metric spaces with 
heat kernels satisfying a sub-Gaussian estimate and a geometric non-negative curvature condition, which would in turn
imply the validity of $wBE(\kappa)$ on these fractafolds. This is beyond the scope of the present paper. 

One long term goal of this project is also to develop tools for Li-Yau type estimates, in particular on the decay of the gradient of the heat kernel. In our setting a gradient is to be understood in a measure-theoretic sense, which motivates a large part of our work. Consideration of fractals in this context is important, in particular, because they appear as models for manifolds with slow heat kernel decay \cite{BCG01}, and limit sets of Schreier graphs of self-similar groups, which include groups of intermediate growth and non elementary amenable groups, see \cite{NT08,Ne05,BGN01,KSW12} and references therein. 
The forthcoming papers~\cite{ABCRST5,ABCRST6} will extend these ideas to non-local forms and infinite dimensional spaces.
 
\subsection*{Main results on BV functions under the $wBE(\kappa)$ condition}

On a $d_H$-Ahlfors regular metric measure space $(X,d,\mu)$  whose heat kernel satisfies the sub-Gaussian estimates \eqref{eq:subGauss-upper},  we define
\[
BV(X):=KS^{\lambda^\#_1,1}(X)=\B^{1,\alpha_1^\#}(X)
\]
and for $f \in BV(X)$,
\[
\mathbf{Var} (f):=\liminf_{r\to 0^+}\int_X\int_{B(x,r)}\frac{|f(y)-f(x)|}{r^{\lambda_1^\#} \mu(B(x,r))}\, d\mu(y)\, d\mu(x).
\]
We show that  for nested fractals, or their products, $BV(X)$ is  dense in $L^1(X,\mu)$, see Theorems~\ref{T:nested} and~\ref{T:nested product}.  

A set $E \subset X$ will be said to be of finite perimeter if $\mathbf{1}_E \in BV(X)$. For a set $E$ of finite perimeter 
we define its perimeter as $P(E)=\mathbf{Var} (\mathbf{1}_E)$.  Note that, unlike in the strictly local setting in \cite{ABCRST2}, the perimeter $P(E)$ may not be induced by a Radon measure in a classical sense, but in some generalized sense that will be further studied in some specific situations in~\cite{ABCRST4}.

Our main assumption to study the BV class is that $X$ satisfies $wBE(\kappa)$ with  $$\kappa=d_W-\lambda_1^\#=d_W(1-\alpha_1^\#).$$  From Theorems \ref{thm-xBE-FD} and \ref{tensorization} this assumption is, for instance, satisfied for nested fractals or their products. The main results we obtain under this assumption are the following:

\begin{enumerate}

\item \underline{Locality property} (Theorem \ref{bounded embedding}):
There is a constant $C>0$ such that for every $f \in BV(X)$,
\[
\sup_{r >0}  \frac{1}{r^{d_H+d_W-\kappa}} \int_X \int_{B(y,r)} |f(x) -f(y)| d\mu(x) \, d\mu(y) \le C \mathbf{Var} (f).
\]
\item  \underline{Co-area estimate} (Theorem \ref{coarea formula 1a}):
There exist constants $c,C>0$ such that for every non-negative $ f \in BV(X)$, 
\[
 c\int_0^\infty \mathbf{Var} ( \mathbf{1}_{E_t(f)}) dt \le \mathbf{Var} (f) \le C \int_0^\infty \mathbf{Var} ( \mathbf{1}_{E_t(f)}) dt,
\]
where $E_t(f)=\{x\in X\, :\, f(x)>t\}$. In particular, for $f\in BV(X)$ the sets $E_t(f)=\{x\in X\, :\, f(x)>t\}$ are
of finite perimeter for almost every $t>0$. 
\item (Theorem \ref{Minkowski}): There exists a constant $C>0$ such that for every Borel set $E \subset X$,
\[
P(E) \le C \mathcal{C}^*_{d_W-\kappa} (E),
\]
where $\mathcal{C}^*_{d_W-\kappa} (E)$ denotes the $(d_W-\kappa)$-codimensional lower Minskowski content of $E$. In particular, any set whose measure-theoretic boundary has finite $(d_W-\kappa)$-codimensional lower Minskowski content has finite perimeter

\item \underline{Sobolev inequality I} (Theorem \ref{Sobolev global}):
Assume $d_W-\kappa < d_H$. Then $BV(X) \subset L^{1^*}(X,\mu)$ and there is $C>0$ such that for every $f \in BV(X)$,
\[
\| f \|_{L^{1^*}(X,\mu)} \le C \mathbf{Var}(f),
\]
where the critical Sobolev exponent $1^*$ is given by the formula
\[
\frac{1}{1^*}=1-\frac{d_W-\kappa}{d_H}.
\]
In particular, there exists a constant $C>0$ such that for every set $E$ of finite perimeter
\[
\mu(E)^{\frac{d_H-d_W+\kappa}{d_H}} \le C P(E).
\]
This is our analog of an isoperimetric inequality.
\item \underline{Sobolev inequality II} (Theorem \ref{Sobolev 2}):
Assume $\kappa=d_W-d_H>0$. Then $BV(X) \subset L^\infty(X,\mu)$ and there exists a constant $C>0$ such that for every $f \in BV(X)$ and a.e. $x,y \in X$
\[
| f(x)-f(y)| \le C \mathbf{Var}(f).
\]
Note that this is in contrast to the strictly local case, where no such pointwise control can be obtained for BV functions; however,
in that case, if $X$ also supports a $1$-Poincar\'e inequality, then we have a pointwise control in terms of the Hardy-Littlewood
maximal function of the BV energy measure.
\end{enumerate}


We also show that BV functions naturally induce Radon measures on $X$ that we call BV measures, see Section \ref{S:BV-measures}. In a certain sense, those measures can be thought of as gradient measures of  BV functions. 
Because of possible oscillatory phenomena due to the geometry of the underlying space $X$, we do not expect that 
a given $f\in BV(X)$ 
has in general a unique associated BV measure. However, Theorem~\ref{equivalence BV} shows 
the remarkable fact that 
all the BV measures associated to a given $f$ are mutually equivalent. If the function $f$ is regular enough we show in Theorem~\ref{thm-bv-kappa} that its energy measure can be controlled by the lower envelope of its BV measures.
 
\subsection*{Main examples}

The motivation for this paper comes from the following
three standard fractal examples: 
unbounded Vicsek set (Figure~\ref{fig-Vicsek}), unbounded Sierpinski gasket (Figure~\ref{fig-SG}), 
and unbounded Sierpinski carpet (Figure~\ref{fig-SC}). The properties of BV spaces are remarkably 
different in these cases and, therefore, on spaces with sub-Gaussian heat kernel bounds~\eqref{eq:subGauss-upper} 
one can expect a theory of BV functions that 
in general is analogous to the BV theory in $\mathbb R^d$, 
but in some sense richer in detail and more variable. 

  On the Vicsek set, as on all nested fractals, 
  $\kappa=d_W-d_H$  and BV functions of finite energy are dense in $L^1$, see
 Theorems~\ref{T:nested} and~\ref{T:Vicsek-SG}(\ref{T:Vicsek-SG1}), with equivalent BV and energy measures. Furthermore, in a future work on the Vicsek set one will see that we can develop a complete theory analogous to the one dimensional case but including new oscillatory phenomena. 
\begin{figure}[htb]\centering
\includegraphics[trim={63 83 180 201},clip,height=0.20\textwidth]{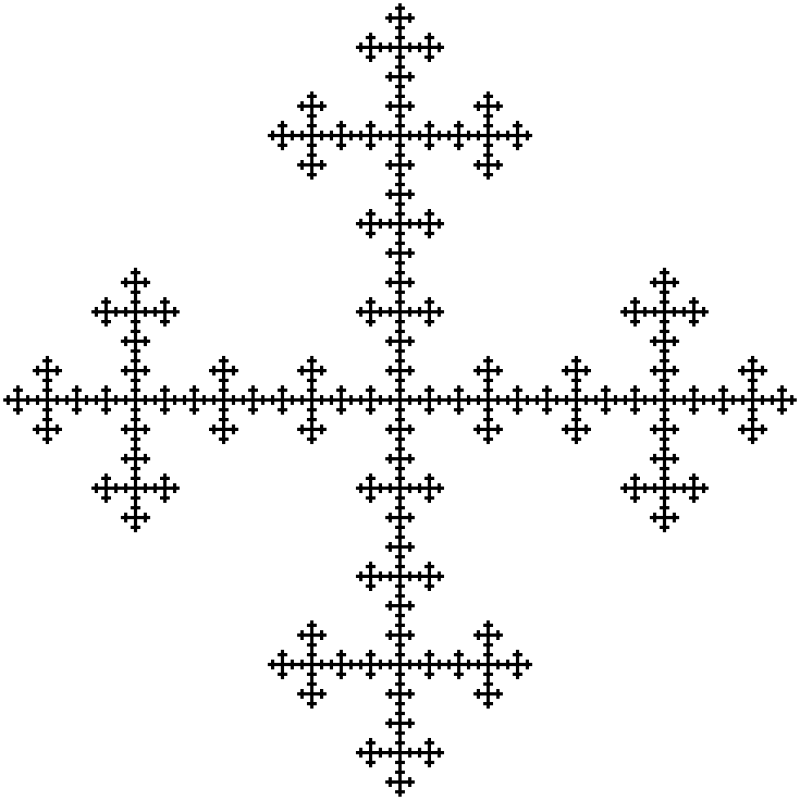}
\caption{A part of an infinite, or unbounded, Vicsek set.} 
\label{fig-Vicsek}
\end{figure}
 \begin{figure}[htb]\centering
 	\includegraphics[trim={83 0 169 63},clip,height=0.20\textwidth]{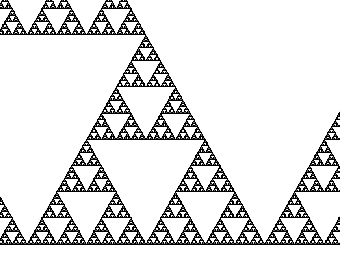}
 	\caption{A part of an infinite, or unbounded, Sierpinski gasket.} \label{fig-SG}
 \end{figure}
 
 On the Sierpinski gasket we also have $\kappa=d_W-d_H$, but BV functions do not even contain piecewise harmonic functions, see Theorems \ref{T:nested} and \ref{T:Vicsek-SG}(\ref{T:Vicsek-SG2}), and one expects all BV functions to be discontinuous, and BV measures to be purely atomic, see Conjecture~\ref{conj-SG}. 
The absence of intrinsically smooth functions of bounded variation is a very surprising phenomenon that has not been observed before. On the Sierpinski gasket differentiability properties of 
intrinsically smooth functions rely on delicate results following from the Furstenberg-Kesten theory of invariant measures and Lyapunov exponents for products of i.i.d. random matrices, which includes a non-commutative matrix version of the classical ergodic theorems. In particular, one can expect that the estimates of the Lyapunov exponents are intimately related to the Besov-type estimates of the intrinsically smooth functions. A detailed analysis involves the  mutual singularity of energy and Hausdorff measures \cite{Kus89,BenBassatStrichartzTeplyaev,HinoS1,HinoS2}. The difference in the analysis  on the Vicsek set and on the Sierpinski gasket is topological, see \cite{ITR12}. 
This will be the subject of future work in~\cite{ABCRST4}.

  \begin{figure}[htb]
  	\centering
  	\includegraphics[trim={8 23 38 70},clip,height=.20\textwidth] {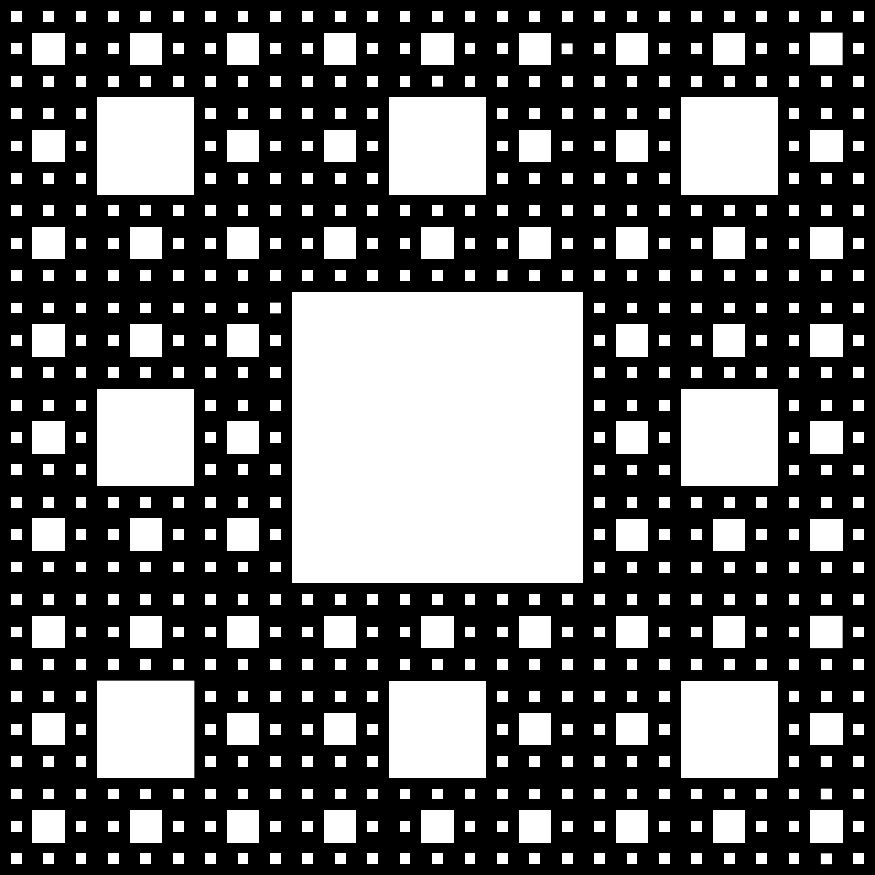}
  	\caption{A part of an infinite, or unbounded, Sierpinski carpet.}\label{fig-SC}
  \end{figure}

On the Sierpinski carpet we 
conjecture that $d_W-d_H <\kappa\leqslant d_W-d_H+d_{tH}-1$, c.f.~Conjecture~\ref{conj-SC}. 
Proving this fact would involve, in particular, improving estimates on the H\"older continuity of harmonic functions obtained in \cite{BB89,Ba98,BB99}. 
It is also
closely related 
to the Besov critical exponents defined in~\eqref{eqn:defnBesovcritexponents}, and to a measurable version of isoperimetric type arguments that will be the subject of future work. 
In addition, this conjecture combines the fractal analog of the Einstein-type relation $r\simeq t^{1/d_W}$ between spatial distance $r$ and time $t$ with a detailed analysis of the $\varepsilon$-neighborhoods of boundaries of open sets 
and our co-area formulas in Section~\ref{S:Co-area formula}.

\subsection*{Convention and notations} 
Throughout the paper we use  $c,C$ to denote positive constants that may change from line to line. The 
quadruple $(X, \mu, \mathcal{E}, \mathcal{F})$ denotes a topological measure space $(X,\mu)$ equipped with a Dirichlet form $(\mathcal{E},\mathcal{F})$ on $L^2(X,\mu)$, where $\mathcal{F}$ denotes the collection of functions $f\in L^2(X,\mu)$ for which $\mathcal{E}(f,f)$ is finite. The notations 
$d_H, d_W$ and $d_{tH}$ denote the Hausdorff, walk, and topological-Hausdorff dimensions respectively, of a space $X$ with metric $d$. 
Strictly speaking, we deal with a metric measure space $(X,d,\mu)$ equipped with a 
Dirichlet form $(\mathcal{E},\mathcal{F})$. We will assume throughout the paper that $\mathcal{E}$ is associated with the metric $d$ via
the sub-Gaussian estimates~\eqref{eq:subGauss-upper} and the weak Bakry-\'Emery curvature conditions~\eqref{WBECDintro}
mentioned above.

\subsection*{Acknowledgments} 
The authors thank 
Prof.\,\,Martin Barlow for helpful information and 
Prof.\,\,Naotaka Kajino for many stimulating and helpful discussions. 

P.A-R. was partly supported by the Feodor Lynen Fellowship, Alexander von Humboldt Foundation (Germany) and the grant DMS~\#1951577. F.B. was partly supported by the grant DMS~\#1660031 of the NSF (U.S.A.) and a Simons Foundation Collaboration grant. L.R. was partly supported by the grant DMS~\#1659643  of the NSF (U.S.A.).  N.S. was partly supported by the grant DMS~\#1800161 of the NSF (U.S.A.).  A.T. was partly supported by the grant DMS~\#1613025 of the NSF (U.S.A.).


\section{Preliminaries}\label{S:Preliminaries}

Our assumptions are quite general, and the main classes of examples we are interested in this paper are fractal spaces. 
We refer to \cite{Ba98,Gri,KigB,Kig:RFQS} for further details on the following framework and assumptions.

\subsection{Metric measure Dirichlet spaces with sub-Gaussian heat kernel estimates}

Let $(X,d,\mu)$ be a locally compact metric measure space where $\mu$ is a Radon measure supported on $X$. Let now $(\mathcal{E},\mathcal{F}=\mathbf{dom}(\mathcal{E}))$ be a Dirichlet form on $X$, that is: a densely defined, closed, symmetric and Markovian form on $L^2(X,\mu)$, see~\cite{FOT,ChenFukushima}. We denote by $C_c(X)$ the vector space of all continuous functions with compact support in $X$ and $C_0(X)$ its 
closure with respect to the supremum norm. A core for $(X,\mu,\mathcal{E},\mathcal{F})$ is a subset 
$\mathcal{C}$ of $C_c(X) \cap \mathcal{F}$ which is dense in $C_c(X)$ in the supremum 
norm and dense in $\mathcal{F}$ in the norm 
\[
\|f\|_{\DF_1}:=
\left( \| f \|_{L^2(X,\mu)}^2 + \mathcal{E}(f,f) \right)^{1/2}.
\]
In the literature, this norm is sometimes denoted by $H^1(X,\mu)$ or $W^{1,2}(X,\mu)$.
The Dirichlet form $\mathcal{E}$ is called regular if it admits a core. It is called  \textit{strongly local} 
if for any two functions $u,v\in\mathcal{F}$ with compact supports such that $u$ is constant in a 
neighborhood of the support of $v$, we have $\mathcal{E}(u,v)=0$, see \cite[Page 6]{FOT}. 
We denote by $\{P_t\}$ the heat semigroup associated with the Dirichlet space $(X,\mathcal{E},\mathcal F, \mu)$ and refer to Section 2.2 in \cite{ABCRST1} for a summary of its basic properties.

Throughout the paper, we make the following assumptions.  Note that they are not independent (some may be derived from combinations of the others); the list was chosen for comprehension rather than minimality.

\begin{assumption}[Regularity]\label{A1}\ 
\begin{itemize}
\item $B(x,r):=\{y\in X\mid d(x,y)<r\}$ has compact closure for any $x\in X$ and any $r\in(0,\infty)$.
\item $\mu$ is Ahlfors $d_H$-regular, i.e.\ there exist $c_{1},c_{2},d_{H}\in(0,\infty)$ such that $c_{1}r^{d_{H}}\leq\mu\bigl(B(x,r)\bigr)\leq c_{2}r^{d_{H}}$ for any $x\in X$ and any $r\in\bigl[0,+\infty\bigr)$.
\item $\mathcal{E}$ is a regular, strongly local Dirichlet form.
\end{itemize}\end{assumption}
\begin{assumption}[Sub-Gaussian Heat Kernel Estimates]\label{sGKHE}
$\{P_t\}$ has a continuous heat kernel $p_t(x,y)$ satisfying, for some $c_{3},c_{4}, c_5, c_6 \in(0,\infty)$ and $d_{W}\in (2,+\infty)$,
 \begin{equation}\label{eq:subGauss-upper}
 c_{5}t^{-d_{H}/d_{W}}\exp\biggl(-c_{6}\Bigl(\frac{d(x,y)^{d_{W}}}{t}\Bigr)^{\frac{1}{d_{W}-1}}\biggr) 
 \le p_{t}(x,y)\leq c_{3}t^{-d_{H}/d_{W}}\exp\biggl(-c_{4}\Bigl(\frac{d(x,y)^{d_{W}}}{t}\Bigr)^{\frac{1}{d_{W}-1}}\biggr)
 \end{equation}
 for $\mu\!\times\!\mu$-a.e.\ $(x,y)\in X\times X$ and each $t\in\bigl(0,+\infty\bigr)$. 
\end{assumption}

The parameter $d_H$ is the Hausdorff dimension, and the parameter $d_W$ is called the walk dimension, even though it is, strictly speaking, not a dimension of a geometric object. 
It is possible to prove that if the metric space $(X,d)$ satisfies a chain condition, then $ 2 \le d_W \le d_H+1$, 
see \cite{GK08,Barlow04,GHL:TAMS2003}. When $d_W=2$, one speaks of Gaussian estimates and when $d_W > 2$, one speaks then of sub-Gaussian estimates. In this framework, it is known that the semigroup $\{P_t\}$ is conservative, i.e. $P_t 1=1$. It is important to note that in this paper, unlike \cite{ABCRST2}, the distance $d(x,y)$ is not necessarily the intrinsic distance associated with the Dirichlet form. The only link between the distance and the Dirichlet form we need to develop our theory is the sub-Gaussian heat kernel estimates \eqref{eq:subGauss-upper}.

The following is an easy but frequently-used consequence of the sub-Gaussian bounds for the heat kernel.
\begin{lemma}
Let $\kappa \ge 0$. There are constants $C,c>0$ so that for every $x,y \in X$ and $t>0$,
\begin{equation}\label{eq:boundpolyptbypct}
 d(x,y)^\kappa p_t(x,y)\leq C t^{\kappa/d_W} p_{ct}(x,y). 
\end{equation}
\end{lemma}
\begin{proof}
Using the sub-Gaussian estimate~\eqref{eq:subGauss-upper} to bound $p_t(x,y)$ from above and $p_{ct}(x,y)$ from below, this reduces to the observation that
\begin{equation*}
 \Bigl(\frac{d(x,y)^{d_W}}t \Bigr)^{\kappa/d_W}\exp\biggl(-c_4 \Bigl(\frac{d(x,y)^{d_W}}t \Bigr)^{\frac1{d_W-1}}\biggr)
 \leq C \exp\biggl(-c_6 \Bigl(\frac{d(x,y)^{d_W}}{ct} \Bigr)^{\frac1{d_W-1}}\biggr)
 \end{equation*} 
for a suitable choice of $C$ and $c$.
\end{proof}

\subsection{Barlow's fractional metric spaces and fractional diffusions}\label{ssec:Barlowspaces}
Our setting is closely related to the setting of fractional metric spaces with fractional diffusions defined by Barlow~\cite{Ba98}, which differ from ours only in that they are assumed to be geodesic; later we will see that this latter assumption implies, and is perhaps no stronger than, the weak Bakry-\'Emery assumption we need in order to construct a rich theory of BV functions.

According to \cite[Definition 3.2]{Ba98}, a complete metric measure space $(X,d)$ with a Borel measure $\mu$ is a \textit{fractional metric space} of dimension $d_H$ if it  is $d_H$-Ahlfors regular and satisfies the \textit{midpoint property}, i.e. for any $x,y\in X$ there exists $z\in X$ such that $d(x,z)=d(z,y)=\frac{1}{2}d(x,y)$.  The latter is equivalent to requiring the space be geodesic.  In this context, Barlow introduced in~\cite[Section 3]{Ba98} a class of processes called \textit{fractional diffusions}. A fractional diffusion is a $\mu$-symmetric, conservative Feller diffusion on $X$ for which there is a jointly continuous heat kernel $p_t(x,y)$ that is symmetric, has the semigroup property and satisfies the sub-Gaussian estimates~\eqref{eq:subGauss-upper}.  In his notation, our $P_t$ is in $\text{FD}(d_H,d_W)$ provided $X$ is geodesic.  We note that in this context~\cite[Theorem 3.20]{Ba98} says $2\leq d_W\leq 1+d_H$.

\subsection{Heat kernel based Besov classes}

 Let $p \ge 1$ and $\alpha \ge 0$. As in \cite{ABCRST1}, we define the Besov seminorm
\begin{equation}\label{eq:defnofheatbesovnorm}
\| f \|_{p,\alpha}:= \sup_{t >0} t^{-\alpha} \left( \int_X \int_X |f(x)-f(y) |^p p_t (x,y) d\mu(x) d\mu(y) \right)^{1/p}
\end{equation}
and define the heat semigroup-based Besov class by
\[
\mathbf{B}^{p,\alpha}(X):=\{ f \in L^p(X,\mu)\, :\, \| f \|_{p,\alpha} <+\infty \}.
\]

Our first goal is to compare the space $\mathbf{B}^{p,\alpha}(X)$ to Besov type spaces previously 
considered in a similar framework (see \cite{Gri}).

For $\alpha\in[0,\infty)$ and $p\in[1,\infty)$, we introduce the following seminorm: for $f\in L^{p}(X,\mu)$ and $r\in(0,\infty)$,
\begin{equation}\label{eq:Besov-seminorm-r}
N^{\alpha}_{p}(f,r):=\frac{1}{r^{\alpha+d_{H}/p}}\biggl(\iint_{\Delta_r}|f(x)-f(y)|^{p}\,d\mu(x)\,d\mu(y)\biggr)^{1/p}
\end{equation}
 and
\begin{equation}\label{eq:Besov-seminorm}
N^{\alpha}_{p}(f):=\sup_{r\in(0,1]}N^{\alpha}_{p}(f,r),
\end{equation}
where $$\Delta_r=\{(x,y)\in X\times X:d(x,y)<r\}.$$

We then define the Besov space $\mathfrak{B}^{\alpha}_{p}(X)$ by
\begin{equation}\label{eq:Besov}
\mathfrak{B}^{\alpha}_{p}(X):=\bigl\{f\in L^{p}(X,\mu)\, :\, N^{\alpha}_{p}(f)<\infty\bigr\}.
\end{equation}

It is clear that $\mathfrak{B}^{\alpha_{2}}_{p}(X)\subset\mathfrak{B}^{\alpha_{1}}_{p}(X)$
for $\alpha_{1},\alpha_{2}\in[0,\infty)$ with $\alpha_{1}\leq\alpha_{2}$.

\begin{theorem}\cite[Theorem~3.2]{P-P10}\label{Besov characterization}
Let $p \ge 1$ and $\alpha \ge 0$. We have $\mathfrak{B}^{\alpha}_{p}(X) = \mathbf{B}^{p,\frac{\alpha}{d_W}}(X)$ and 
there exist constants $c_{p},C_{p,\alpha}>0$ such that for every $f \in \mathfrak{B}^{\alpha}_{p}(X)$ and $r >0$,
\[
c_{p} \sup_{s\in(0,r]}N^{\alpha}_{p}(f,s) \le \| f \|_{p,\alpha/d_W} 
\le C_{p,\alpha} \left( \sup_{s\in(0,r]}N^{\alpha}_{p}(f,s)+\frac{1}{r^{\alpha}} \|f\|_{L^{p}(X,\mu)} \right).
\]
In particular, $ \| f \|_{p,\alpha/d_W} \simeq \sup_{s\in(0,+\infty) }N^{\alpha}_{p}(f,s)$, and $f\in L^p(X,\mu)$ is in $\B^{p,\frac\alpha{d_W}}(X)$ if and only if $\limsup_{r\to0^+} N_p^\alpha(f,r)<\infty.$
\end{theorem}

\begin{remark}
The above theorem is essentially a rephrasing of~\cite[Theorem~3.2]{P-P10}, though the estimates we obtain are slightly sharper. However, 
the notion of Besov spaces given in~\cite{P-P10} considers dyadic 
jumps in the parameter $t$; hence the proof given there is slightly more complicated than ours. We include the relatively short proof because the estimates are useful in later sections.
\end{remark}
\begin{proof}
We first prove the lower bound. For $t>0$ and $\alpha>0$ observe from~\eqref{eq:subGauss-upper} that
\begin{equation*}
	p_t(x,y)\geq c_5 \exp(-c_6) t^{-d_H/d_W}   \quad\text{for}\quad d(x,y)\leq t^{1/d_W},
	\end{equation*}
so that
\begin{align*}
 \| f \|^p_{p,\alpha/d_W}
 &= \sup_{t>0} t^{-\frac{\alpha p}{d_W}} \int_X \int_X |f(x)-f(y) |^p p_t (x,y) d\mu(x) d\mu(y)\\
 &\geq \sup_{t>0} t^{-\frac{\alpha p}{d_W}} \iint_{\Delta_{t^{1/d_W}}} |f(x)-f(y) |^p p_t (x,y) d\mu(x) d\mu(y)\\
 &\geq  \sup_{t>0} c_5 \exp(-c_6) \frac1{t^{(\alpha p+d_H)/d_W}} \iint_{\Delta_{t^{1/d_W}}} |f(x)-f(y) |^p d\mu(x) d\mu(y)\\
 &=c_5 \exp(-c_6)  \sup_{t>0} N_p^\alpha(f,t^{1/d_W})^p,
\end{align*}
from which the lower bound follows.
We now turn to the upper bound. Fixing $r>0$, we set
\begin{align}
A(t)&:=\int_X\int_{X\setminus B(y,r)}p_{t}(x,y)|f(x)-f(y)|^{p}\,d\mu(x)\,d\mu(y),\label{E:A(t)}\\
B(t)&:=\iint_{\Delta_r}p_{t}(x,y)|f(x)-f(y)|^{p}\,d\mu(x)\,d\mu(y),\label{E:B(t)}
\end{align}
so that $ \int_X \int_X |f(x)-f(y) |^p p_t (x,y) d\mu(x) d\mu(y) =A(t)+B(t)$.
By \eqref{eq:subGauss-upper} and the inequality
$|f(x)-f(y)|^{p}\leq 2^{p-1}(|f(x)|^{p}+|f(y)|^{p})$,

\begin{align}
A(t)&\leq\frac{c_{3}}{t^{d_{H}/d_{W}}}\int_X\int_{X\setminus B(y,r)}\exp\biggl(-c_{4}\Bigl(\frac{d(x,y)^{d_{W}}}{t}\Bigr)^{\frac{1}{d_{W}-1}}\biggr)\cdot 2^{p}|f(y)|^{p}\,d\mu(x)\,d\mu(y)
\notag\\
&=
\frac{2^{p}c_{3}}{t^{d_{H}/d_{W}}}\sum_{k=1}^{\infty}\int_X\int_{B(y,2^{k}r)\setminus B(y,2^{k-1}r)}\exp\biggl(-c_{4}\Bigl(\frac{d(x,y)^{d_{W}}}{t}\Bigr)^{\frac{1}{d_{W}-1}}\biggr)|f(y)|^{p}\,d\mu(x)\,d\mu(y)
\notag\\
&\leq
\frac{2^{p}c_{3}}{t^{d_{H}/d_{W}}}\sum_{k=1}^{\infty}\int_X\mu\bigl(B(y,2^{k}r)\bigr)\exp\biggl(-c_{4}\Bigl(\frac{2^{(k-1)d_{W}}r^{d_{W}}}{t}\Bigr)^{\frac{1}{d_{W}-1}}\biggr)|f(y)|^{p}\,d\mu(y)
\notag\\
&\leq
\frac{2^{p}c_{3}}{t^{d_{H}/d_{W}}}\sum_{k=1}^{\infty}c_{2}r^{d_{H}}2^{kd_{H}}\|f\|_{L^{p}}^{p}\exp\biggl(-c_{4}\Bigl(\frac{r^{d_{W}}}{t}\Bigr)^{\frac{1}{d_{W}-1}}\Bigl(2^{\frac{d_{W}}{d_{W}-1}}\Bigr)^{k-1}\biggr)
\notag\\
&=
\|f\|_{L^{p}}^{p}2^{p}c_{2}c_{3}\sum_{k=1}^{\infty}2^{d_{H}}\Bigl(\frac{r^{d_{W}}}{t}2^{d_{W}(k-1)}\Bigr)^{d_{H}/d_{W}}\exp\biggl(-2^{-\frac{d_{W}}{d_{W}-1}}c_{4}\Bigl(\frac{r^{d_{W}}}{t}2^{d_{W}k}\Bigr)^{\frac{1}{d_{W}-1}}\biggr)
\notag\\
&\leq
\|f\|_{L^{p}}^{p}2^{p+d_{H}}c_{2}c_{3}\sum_{k=1}^{\infty}\int_{(r^{d_{W}}/t)(2^{d_{W}})^{k-1}}^{(r^{d_{W}}/t)(2^{d_{W}})^{k}}s^{d_{H}/d_{W}}\exp\Bigl(-2^{-\frac{d_{W}}{d_{W}-1}}c_{4}s^{\frac{1}{d_{W}-1}}\Bigr)\frac{1}{(d_{W}\log 2)s}\,ds
\notag\\
&=
\frac{2^{p+d_{H}}c_{2}c_{3}}{d_{W}\log 2}\|f\|_{L^{p}}^{p}\int_{r^{d_{W}}/t}^{\infty}s^{d_{H}/d_{W}-1}\exp\Bigl(-2^{-\frac{d_{W}}{d_{W}-1}}c_{4}s^{\frac{1}{d_{W}-1}}\Bigr)\,ds\notag\\
&\leq c_{8}\exp\biggl(-c_{9}\Bigl(\frac{r^{d_{W}}}{t}\Bigr)^{\frac{1}{d_{W}-1}}\biggr)\|f\|_{L^{p}(X,\mu)}^{p},
\label{eq:HKBesov-norms-upper-proof1}
\end{align}
where $c_{9}:=2^{-\frac{d_{W}}{d_{W}-1}}c_{4}$ and
$c_{8}:=2^{p+d_{H}}c_{2}c_{3}(d_{W}\log 2)^{-1}\int_{0}^{\infty}s^{d_{H}/d_{W}-1}\exp\bigl(-c_{9}s^{\frac{1}{d_{W}-1}}\bigr)\,ds$.

On the other hand, for $B(t)$, by \eqref{eq:subGauss-upper} we have
\begin{align}
B(t)&\leq\frac{c_{3}}{t^{d_{H}/d_{W}}}\iint_{\Delta_{r}} \exp\biggl(-c_{4}\Bigl(\frac{d(x,y)^{d_{W}}}{t}\Bigr)^{\frac{1}{d_{W}-1}}\biggr)|f(x)-f(y)|^{p}\,d\mu(x)\,d\mu(y)\notag\\
&\leq
\frac{c_{3}}{t^{d_{H}/d_{W}}}\sum_{k=1}^{\infty}\int\limits_X\int\limits_{B(y,2^{1-k}r)\setminus B(y,2^{-k}r)}\exp\biggl(-c_{4}\Bigl(\frac{d(x,y)^{d_{W}}}{t}\Bigr)^{\frac{1}{d_{W}-1}}\biggr)|f(x)-f(y)|^{p}\,d\mu(x)\,d\mu(y)\notag\\
&\leq
\frac{c_{3}}{t^{d_{H}/d_{W}}}\sum_{k=1}^{\infty}\iint_{\Delta_{2^{1-k}r}}\exp\biggl(-c_{4}\Bigl(\frac{2^{-kd_{W}}r^{d_{W}}}{t}\Bigr)^{\frac{1}{d_{W}-1}}\biggr)|f(x)-f(y)|^{p}\,d\mu(x)\,d\mu(y)\notag\\
&\leq 
c_{3} \sum_{k=1}^{\infty}\frac{(2^{1-k}r)^{p\alpha+d_{H}}}{t^{d_{H}/d_{W}}}\exp\biggl(-c_{9}\Bigl(\frac{r^{d_{W}}}{t}2^{d_{W}(1-k)}\Bigr)^{\frac{1}{d_{W}-1}}\biggr)
\iint_{\Delta_{2^{1-k}r}}
\frac{ |f(x)-f(y)|^{p} }{ (2^{1-k}r)^{p\alpha+d_{H}} }
\,d\mu(x)\,d\mu(y)\notag\\
&\leq
 c_{3}2^{p\alpha+d_{H}} t^{\frac{p\alpha}{d_W} }\sup_{s\in(0,r]}N^{\alpha}_{p}(f,s)^{p}\sum_{k=1}^{\infty}\Bigl(\frac{r^{d_{W}}}{t}2^{-d_{W}k}\Bigr)^{\frac{d_H+p\alpha}{d_W}}\exp\biggl(-c_{9}\Bigl(\frac{r^{d_{W}}}{t}2^{-d_{W}(k-1)}\Bigr)^{\frac{1}{d_{W}-1}}\biggr)\notag\\
&\leq 
c_{7}t^{\frac{p\alpha}{d_W} }\sup_{s\in(0,r]}N^{\alpha}_{p}(f,s)^{p}.
\label{eq:HKBesov-norms-upper-proof2}
\end{align}
Hence one has 
\[
 \int_X \int_X |f(x)-f(y) |^p p_t (x,y) d\mu(x) d\mu(y) 
 \le c_{8}\exp\biggl(-c_{9}\Bigl(\frac{r^{d_{W}}}{t}\Bigr)^{\frac{1}{d_{W}-1}}\biggr)\|f\|_{L^{p}}^{p}
 +c_{7}t^{\frac{p\alpha}{d_W} }\sup_{s\in(0,r]}N^{\alpha}_{p}(f,s)^{p}.
\]
This yields
\[
\sup_{t>0}t^{-\frac{p\alpha}{d_W}}\int_X \int_X |f(x)-f(y) |^p p_t (x,y) d\mu(x) d\mu(y) 
\le c_{7} \sup_{s\in(0,r]}N^{\alpha}_{p}(f,s)^{p}+\frac{c_{10}}{r^{p\alpha}} \|f\|_{L^{p}(X,\mu)}^{p}.
\]
The proof is thus complete.
\end{proof}

\subsection{Besov regularity of indicators of sets and density of $\B^{1,\alpha}$ in $L^1$}\label{ssec:besovdensity}

In order to have an interesting theory we certainly need  our Besov spaces to contain non-constant functions.  It is also natural to be concerned with whether they are dense in the Lebesgue spaces.  Accordingly we follow~\cite[Section~5.2]{ABCRST1} and define critical exponents as follows:
\begin{equation}\label{eqn:defnBesovcritexponents}
\begin{aligned}
\alpha^*_p &=\sup \{ \alpha >0\, :\, \B^{p,\alpha}(X) \text{ is dense in } L^p(X,\mu) \},\\
\alpha^{\#}_p &=\sup \{ \alpha >0\, :\, \B^{p,\alpha}(X) \text{ contains non-constant functions}\}.
\end{aligned}
\end{equation}
Note that $\alpha_p^*\le \alpha_p^\#$.
In this section we concern ourselves only with a simple condition for $\B^{1,\alpha}(X)$ to be dense in $L^1(X,\mu)$, as this is the essential case to consider in studying functions of bounded variation.  It will become apparent when we treat the co-area formula, Theorem~\ref{coarea formula 1a} that the significant question is whether a characteristic function $\mathbf{1}_E$ of a Borel set $E$ is in $\B^{1,\alpha}(X)$, and it is well-known that this is related to boundary regularity of the measure-theoretic boundary.

\begin{definition}
Let $E\subset X$ be a Borel set. We say $x$ is a Lebesgue density point of $E$ and write $x\in E^*$ if
\begin{equation*}
	\limsup_{r\to 0^+}\frac{\mu(B(x,r)\cap E)}{\mu(B(x,r))}>0.
	\end{equation*}
The measure-theoretic boundary is $\partial^*E=E^*\cap(E^c)^*$, see e.g.~\cite[Section 4]{AMP}.  Now for $r>0$ 
 define the measure-theoretic $r$-neighborhood $\partial_r^*E$ by  
\begin{equation}\label{eq:measure r neighborhood}
\partial^*_rE:=\big(E^*\cap(E^c)_r\big)\cup \big((E^c)^*\cap E_r\big),
\end{equation}
where $E_r=\{x\in X:\,\mu(B(x,r)\cap E)>0\}$ and similarly for $(E^c)_r$.
\end{definition} 
Notice that $\partial^*E\subset\cap_{r>0}\partial^*_r E\subset\partial E$, where this last is the topological boundary.  We have the following easy consequence of Theorem~\ref{Besov characterization}.

\begin{lemma}\label{lem:charfunctinBesov}
Suppose $E\subset X$ is a finite measure Borel set such that
\begin{equation*}
	\limsup_{r\to0^+} r^{-\alpha d_W} \mu(\partial^*_r E) <\infty,
	\end{equation*}
then $\mathbf{1}_E\in\B^{1,\alpha}(X)$.
\end{lemma}
\begin{proof}
From Theorem~\ref{Besov characterization} it suffices to show that $\limsup_{r\to0^+}N_1^{\alpha d_W}(\mathbf{1}_E,r)<\infty$, however computing directly from the definition~\eqref{eq:Besov-seminorm-r}
\begin{align*}
	N_1^{\alpha d_W}(\mathbf{1}_E,r)
	&= \frac1{r^{\alpha d_W+d_H}} \iint_{\Delta_r} |\mathbf{1}_E(x)-\mathbf{1}_E(y)|\,d\mu(x)\,d\mu(y) \\
	&= \frac1{r^{\alpha d_W+d_H}} \biggl( \int_E \mu\bigl(B(y,r)\cap E^c\bigr) \,d\mu(y) + \int_{E^c} \mu\bigl(B(y,r)\cap E\bigr)\,d\mu(y) \biggr)\\
	&= \frac1{r^{\alpha d_W+d_H}} \biggl( \int_{E^*\cap (E^c)_r} \mu\bigl(B(y,r)\cap E^c\bigr) \,d\mu(y) + \int_{(E^c)^*\cap E_r} \mu\bigl(B(y,r)\cap E\bigr)\,d\mu(y) \biggr)\\
	&\leq 2c_2r^{-\alpha d_W} \mu(\partial^*_rE),
	\end{align*}
where we used Ahlfors regularity both to ensure Lebesgue density points have full measure (see~\cite[Lemma~1.8]{Hein-Lect} or
\cite[Section 3.4]{nages_book}) so as to pass from $E$ to $E^*$ and from $E^c$ to $(E^c)^*$, and for the estimate $\mu(B(y,r))\leq c_2r^{d_H}$.
\end{proof}
Note that in the above lemma we must have $\alpha d_W\leq d_H$ because the measure of the $r$-neighborhood of a point is bounded below by $cr^{d_H}$.  Rephrasing the above in terms of the critical exponents and using the density in $L^1(X,\mu)$ of the set of characteristic  functions of a basis for the topology we have the following result.

\begin{corollary}\label{cor:conditforbesovdensity}
If there is a non-empty open set $E$ with $\mu(E)<\infty$ and $\limsup_{r\to0^+}r^{-\alpha d_W}\mu(\partial_r^*E)<\infty$, 
then $\alpha_1^\#\geq \alpha$.  If this estimate is true for a family of open sets $E$ which generates the 
topology of $X$, then $\alpha_1^*\geq\alpha$. 
\end{corollary}

\section{Weak Bakry-\'Emery curvature condition} \label{Section wBE}

In this section we introduce the key condition that will allow us to study the BV class. This condition, called $wBE(\kappa)$, is the H\"older analogue of the weak Bakry-\'Emery curvature condition that was previously introduced in strictly local Dirichlet spaces (see~\cite{ABCRST2,Ba-Ke19}). It quantifies a uniform regularization property of the heat semigroup in a scale invariant manner.

\subsection{Definition of $wBE(\kappa)$}

\begin{definition}
We say that $(X,\mu,\mathcal{E},\mathcal{F})$ satisfies the weak Bakry-\'Emery non-negative curvature condition $wBE(\kappa)$ if there exist a constant $C>0$ and a parameter $0 < \kappa < d_W$ such that for every $t >0$, $g \in L^\infty(X,\mu)$ and  $x,y \in X$,
\begin{align}\label{WBECD}
| P_t g (x)-P_tg(y)| \le C \frac{d(x,y)^\kappa}{t^{\kappa/d_W}} \| g \|_{ L^\infty(X,\mu)}.
\end{align}
\end{definition}

\begin{remark}\label{chain condition}
We note that if $(X,d)$ satisfies a chain condition as in \cite[Section~7]{Gri}, then one must have $\kappa \le 1$. 
The chain condition is that there is a constant $C_h>0$ so that for each $x,y\in X$ and positive integers $n\ge 2$,
there is a sequence of points $x=x_1, x_2,\cdots,x_n=y$ from $X$ such that for $j=1,\cdots, n-1$ we have
$d(x_j,x_{j+1})\le \tfrac{C_h}{n} d(x,y)$.
Indeed, it is easily seen that the chain condition implies that $\alpha$-H\"older functions $g$ with $\alpha >1$ need to be constant,
for then we have for each integer $n\ge 2$,
\[
|g(x)-g(y)|\le \sum_{j=1}^{n-1}|g(x_j)-g(x_{j+1})|
 \le C\, C_h^\alpha\sum_{j=1}^{n-1}n^{-\alpha}d(x,y)^\alpha\le C\, C_h^\alpha\, d(x,y)^\alpha \, n^{1-\alpha},
\]
and as $\alpha>1$, letting $n\to\infty$ shows that $g(x)=g(y)$.
\end{remark}
Since $P_t$ is a contraction in $L^\infty(X,\mu)$, the estimate \eqref{WBECD} is only relevant when $d(x,y)$ is small compared to $t^{1/d_W}$. In particular, one has the following result:

\begin{lemma}\label{non-unique kappa}
If $(X,\mu,\mathcal{E},\mathcal{F})$ satisfies $wBE(\kappa)$ for some $0 < \kappa < d_W$, then it satisfies $wBE(\kappa')$ for every $0 < \kappa' \le \kappa$.
\end{lemma}

\begin{proof}
Since $\bigl(t^{-1/d_W}d(x,y)\bigr)^\kappa\leq\bigl(t^{-1/d_W}d(x,y)\bigr)^{\kappa'}$ if $d(x,y)\leq t^{1/d_W}$, it suffices to consider $d(x,y)>t^{1/d_W}$, for which the result follows from the estimate
\begin{equation*}
| P_t g (x)-P_tg(y)| \le 2 \| P_t g \|_{ L^\infty(X,\mu)}\le 2 \| g \|_{ L^\infty(X,\mu)} \le 2 \frac{d(x,y)^{\kappa'}}{t^{\kappa'/d_W}} \| g \|_{ L^\infty(X,\mu)}.\qedhere
\end{equation*}
\end{proof}

The weak Bakry-\'Emery condition is also related to the H\"older regularity of the heat kernel.

\begin{lemma}\label{Holder kernel}
Assume that $(X,\mu,\mathcal{E},\mathcal{F})$ satisfies the weak Bakry-\'Emery condition $wBE(\kappa)$. Then, there exists a constant $C>0$ such that for every $t>0$, $x,y,z \in X$,
\begin{equation}\label{eqn:wBEimpliesptHolder}
| p_t(x,z)-p_t(y,z)| \le C \frac{d(x,y)^\kappa}{t^{(\kappa+d_H)/d_W}}.
\end{equation}
\end{lemma}

\begin{proof}
The weak Bakry-\'Emery estimate is equivalent to
\[
\int_X | p_t(x,z)-p_t(y,z)| d\mu(z)\le C \frac{d(x,y)^\kappa}{t^{\kappa/d_W}},
\]
so the result follows by computing
\begin{align*}
| p_t(x,z)-p_t(y,z)| &=\left| \int_X (p_{t/2}(x,u)-p_{t/2}(y,u)) p_{t/2}(u,z) d\mu(u)\right| \\
 &\le \frac{C}{t^{d_H/d_W}} \int_X | p_{t/2}(x,u)-p_{t/2}(y,u)| d\mu(u).\qedhere
\end{align*}
\end{proof}

\begin{remark}
Note that if \eqref{eq:subGauss-upper} holds, then there always exists $\kappa\in(0,1)$ such that~\eqref{eqn:wBEimpliesptHolder} holds for every $t>0$, $\mu$-a.e. $x,y,z \in X$, see for instance~\cite[Proposition 4.5]{Coulhon03}, or~\cite[Section 5.3]{HSC01}. However one should not expect \cite{Coulhon03,HSC01} to give the optimal $\kappa$.
\end{remark}

\begin{remark}
Suppose $(X,d,\mu)$ is Ahlfors $d_H$-regular and the heat kernel satisfies the sub-Gaussian upper bound in~\eqref{eq:subGauss-upper}. Then validity of the weak Bakry-\'Emery curvature condition~\eqref{WBECD} implies the sub-Gaussian {\bfseries lower bound} for the heat kernel in~\eqref{eq:subGauss-upper} for every $t>0$, $\mu$-a.e. $x,y \in X$.
This is established using the proof of~\cite[Theorem 3.1]{Coulhon03}, for which one needs to know that $p_t(x,y)\ge ct^{-d_H/d_W}$ on $d(x,y)<at^{1/d_W}$ for some $a>0$. 
In order to validate the latter, we first note that the Ahlfors $d_H$-regularity and heat kernel upper bound imply that for any $t>0$ and $\mu$-a.e. $x\in M$, we have $p_t(x,x)\ge ct^{-d_H/d_W}$. Indeed, this can be obtained by applying the conservativeness property of the heat semigroup and Jensen's inequality. Furthermore,  the weak Bakry-\'Emery condition gives~\eqref{eqn:wBEimpliesptHolder} from which we obtain a suitable $a>0$ via the computation:
\[
 | p_t(x,x)-p_t(y,x)| \le C \frac{d(x,y)^\kappa}{t^{(\kappa+d_H)/d_W}} \le C \frac{(a t^{1/d_W})^\kappa}{t^{\kappa/d_W}} p_t(x,x) \le \frac12 p_t(x,x).
\]
\end{remark}

\subsection{Fractional metric spaces satisfying $wBE(\kappa)$}\label{subsec-frmms}
Recall from Section~\ref{ssec:Barlowspaces} that if $X$ is geodesic then our assumptions put us in Barlow's class of fractional metric spaces with fractional diffusions. Assuming this, and that the dimension of the state space is small compared to the walk dimension, specifically that $d_H<d_W$, we can use Barlow's results to obtain a weak Bakry-\'Emery inequality.  Specifically,~\cite[Theorem 3.40]{Ba98} says that then  for any $\lambda>0$, $x,y\in X$ and $f\in L^\infty(X,\mu)$,
\begin{equation}\label{E:FD_resolvent_estimate}
|U_\lambda f(x)-U_\lambda f(y)|\leq C\lambda^{-\frac{d_H}{d_W}}d(x,y)^{d_W-d_H}\|f\|_{L^\infty(X,\mu)},
\end{equation}
where $\{U_\lambda\}_{\lambda>0}$ is the resolvent associated with $P_t$.  The latter may be written as $(L-\lambda)^{-1}$ using the generator $L$ for which $P_t=e^{tL}$.  From this we deduce $wBE(d_W-d_H)$ as follows.

\begin{theorem}\label{thm-xBE-FD}
A fractional metric space $(X,d,\mu)$ for which  $1\leq d_H <d_W$
 satisfies $wBE(\kappa)$ with $\kappa=d_W-d_H$.
\end{theorem}

\begin{proof}
Let $f\in L^\infty(X,\mu)$ and $\lambda >0$. 
Since the heat semigroup solves the heat equation, we have $LP_t=\frac{\partial P_t}{\partial t}$, and so
\[
|(L-\lambda)P_tf(x)|
\leq \int_X\Big|\frac{\partial}{\partial t}p_t(x,y)-\lambda p_t(x,y)\Big|\,d\mu(y)\|f\|_{L^\infty(X,\mu)}.
\]
Applying~\cite[Corollary 5]{Da97} and the sub-Gaussian estimates in~\eqref{eq:subGauss-upper}, we obtain
\[
\Big|\frac{\partial}{\partial t}p_t(x,y)\Big|
\le  \frac{C}{ t}t^{-\frac{d_H}{d_W}}\exp\Big(-c\Big(\frac{d(x,y)^{d_W}}{t}\Big)^{\frac{1}{d_W-1}}\Big) 
\le  \frac{C}{ t} p_{ct}(x,y),
\]
where the two pairs of constants $c,C$ are different.
By the conservativeness property of the heat semigroup, i.e. $\int_X p_t (x,y) d\mu(y)=1$, we obtain
\[
|(L-\lambda)P_tf(x)|
\le \int_X  \Big(\frac{C}{t}p_{ct}(x,y)+\lambda p_t(x,y)\Big)\,d\mu(y)\|f\|_{L^\infty(X,\mu)}
\leq \Big(\frac{C}{t}+\lambda\Big)\|f\|_{L^\infty(X,\mu)}.
\]
Then, we note that $P_tf=U_\lambda(L-\lambda)P_tf$. Therefore,  in view of~\eqref{E:FD_resolvent_estimate}, for any $\lambda>0$ and $x,y\in X$ we have
\begin{align*}\label{E:FD_01}
|P_tf(x)-P_tf(y)| & \leq C\lambda^{-\frac{d_H}{d_W}}d(x,y)^{d_W-d_H}\|(L-\lambda)P_tf\|_{L^\infty(X,\mu)} \\
 & \leq C\lambda^{-\frac{d_H}{d_W}}d(x,y)^{d_W-d_H}  \Big(\frac{C}{t}+\lambda\Big)\|f\|_{L^\infty(X,\mu)}. \notag
\end{align*}
So choosing $\lambda=C/t$, it follows  that
\begin{equation*}
|P_tf(x)-P_tf(y)|
\le C t^{-\frac{d_W-d_H}{d_W}}d(x,y)^{d_W-d_H}\|f\|_{L^\infty(X,\mu)}.\qedhere
\end{equation*}
\end{proof}

We remark that the proof of~\eqref{E:FD_resolvent_estimate} uses the geodesic property in two places.  In one, the proof of \cite[Corollary~3.38]{Ba98}, one can use Ahlfors regularity instead, but its use in the proof of~\cite[Lemma~3.9]{Ba98} is more subtle.  More significantly, we will see that the weak Bakry-\'Emery estimate $wBE(d_W-d_H)$ is optimal for finitely ramified sets like the Sierpinski gasket or Vicsek set  in the sense that $wBE(\kappa)$ is not satisfied if $\kappa >d_W-d_H$.  However we believe it is not optimal for the Sierpinski carpet, see Conjecture~\ref{conj-SC}.  In view of the discussion in Section~\ref{S:critical_exp_local}, the Vicsek set may be particularly interesting because $\kappa=1$ and the Vicsek set is a dendrite, see \cite{KigamiDendrites,T08}.

\subsection{Stability of $wBE(\kappa)$ by tensorization}\label{subsec-products}

The condition $wBE(\kappa)$ is stable under tensor powers. This yields many examples, and is important in analysis on fractals, see~\cite{IRS13,St09p,St07p,St05p}.

Let $(X, d,\mu)$ satisfy the assumptions \ref{A1} and \ref{sGKHE}.
Consider $(X^n, d_{X^n},\mu^{\otimes n})$, where the metric $ d_{X^n}$ is defined as follows: for any $\mathbf x=(x_1, x_2,\cdots, x_n) \in X^n$ and $\mathbf x'=(x_1', x_2',\cdots, x_n')\in X^n$,
\[
d_{X^n}(\mathbf x, \mathbf x')^{\frac{d_W}{d_W-1}}=\sum_{i=1}^n d_X(x_i,x_i')^{\frac{d_W}{d_W-1}}.
\]
Then $(X^n, d_{X^n},\mu^{\otimes n})$ is Ahlfors $nd_H$-regular and the heat kernel on $X^n$ given by
\[
p_t^{X^n}(\mathbf x, \mathbf y)=p_t(x_1, y_1)\cdots p_t(x_n, y_n)
\]
satisfies the sub-Gaussian bounds~\eqref{eq:subGauss-upper} with $d(x,y)$ replaced by $d(\mathbf x, \mathbf y)$.

\begin{proposition}\label{tensorization}
If $(X,d,\mu)$ satisfies $wBE(\kappa)$, then $(X^n, d_{X^n},\mu^{\otimes n})$ also satisfies $wBE(\kappa)$.
\end{proposition}

\begin{proof}
Let $f\in L^{\infty}(X^n, \mu^{\otimes n})$.   Given $\mathbf x, \mathbf x'\in X^n$ we may estimate the change in $P_t^{X^n}f$ due to changing a single component of $\mathbf x$ using $wBE(\kappa)$ on $X$:
\begin{align*}
&\bigl|P_t^{X^n} f(x_1', \cdots,x_{i-1}',x_{i},\cdots, x_n)-P_t^{X^n} f(x_1', \cdots,x_i',x_{i+1},\cdots, x_n)\bigr|
\\ &\le 
\bigl|P_t(P_t^{X^{n-1}} f(x_1', \cdots,x_{i-1}',\cdot,x_{i+1},\cdots, x_n))(x_i)-P_t(P_t^{X^{n-1}} f(x_1', \cdots,x_{i-1}',\cdot,x_{i+1},\cdots, x_n))(x_i')\bigr|
\\ &\le
C \frac{d(x_i,x_i')^{\kappa}}{t^{\kappa/d_W}} \left\|P_t^{X^{n-1}} f(x_1', \cdots,x_{i-1}',\cdot,x_{i+1},\cdots, x_n)\right\|_{L^{\infty}(X,\mu)} \leq C \frac{d(x_i,x_i')^{\kappa}}{t^{\kappa/d_W}} \|f\|_{L^{\infty}(X^n)}.
\end{align*}
Bounding the change in each component in this manner and summing over components yield
\[
\bigl|P_t^{X^n} f(\mathbf x)-P_t^{X^n} f(\mathbf x')\bigr|
\le C\frac{1}{t^{\kappa/d_W}} \brak{\sum_{i=1}^n d(x_i,x_i')^{\kappa}} \|f\|_{L^{\infty}(X^n)}
\le
C \frac{d_{X^n}(\mathbf x,\mathbf x')^{\kappa}}{t^{\kappa/d_W}} \|f\|_{L^{\infty}(X^n)},
\]
where $C$ is independent of $n$ if $\kappa\geq1$.
\end{proof}

\subsection{Continuity of $P_t$ in $\B^{p,\alpha}$ for $p \ge 2$ and pseudo-Poincar\'e inequalities}

In this section, we study the continuity of the heat semigroup in the Besov spaces in the range $p \ge 2$. This complements the results of \cite[Section 5.1]{ABCRST1} which treated the case $1<p\le 2$ without the weak Bakry-\'Emery estimate~\eqref{WBECD}.
We also obtain the pseudo-Poincar\'e inequalities in Proposition~\ref{pseudo poincare}.  These are a key technical tool in our work, especially in the case $p=1$ which  is an analog of~\cite[Lemma 4.3]{ABCRST2} in the sub-Gaussian setting.

In what follows we use the following notation.  For $1\leq p<\infty$ let
\begin{equation}\label{eq:defnofbetap}
	\beta_p= \Bigl(1-\frac{2}{p}\Bigr)\frac{\kappa}{d_W}+\frac{1}{p}
\end{equation}
and note that the restriction $0<\kappa<d_W$ in~\eqref{WBECD} implies $0<\beta_p<1$.

\begin{theorem}\label{P:PtinBesovp4}
 If $(X,\mu,\mathcal{E},\mathcal{F})$ satisfies the weak Bakry-\'Emery condition $wBE(\kappa)$ and $2\leq p<\infty$ then there is  $C>0$ such that for every $t>0$ and $f\in L^p(X,\mu)$
\[
\| P_t f \|_{p,\beta_p} \le \frac{C}{t^{ \beta_p}}\| f\|_{L^p(X,\mu)}.
\]
In particular, for $t>0$, $P_t: L^p(X,\mu) \to \mathbf{B}^{p,\beta_p}(X)$ is bounded.
\end{theorem}

\begin{proof}

For $r >0$ consider the measure on $X \times X$ given by $d\nu_r (x,y) =1_{d(x,y)<r} d\mu(x) \otimes d\mu(y)$.

From the sub-Gaussian heat kernel lower bound \eqref{eq:subGauss-upper}, we have 
\begin{align*}
\int_X \int_X (P_t f (x) -P_t f (y))^2 d\nu_r (x,y)&  = \int_X \int_{B(x,r)} (P_t f (x) -P_t f (y))^2 d\mu (y) d\mu(x) \\
 & \le C r^{d_H} \int_X \int_{X} p_{r^{d_W}}(x,y) (P_t f (x) -P_t f (y))^2 d\mu (y) d\mu(x).
\end{align*}

 It was proved in~\cite[Proposition~4.6]{ABCRST1} that $\| f\|_{2,1/2}^2=2\mathcal{E}(f,f)$.  Applying this to $P_t f$ we conclude from standard estimates that $\|P_t f\|_{2,1/2}\leq C t^{-1/2}\|f\|_{L^2(X,\mu)}$ (see also ~\cite[Theorem 5.1]{ABCRST1}). This is equivalent to the following bound, valid for all $s>0$:
\begin{equation*}
	\int_X\int_X p_s(x,y)(P_t f(x)-P_tf(y))^2 d\mu(x)\,d\mu(y)
	\leq C\frac{s}{t} \|f\|^2_{L^2(X,\mu)}.
	\end{equation*}
Therefore,
\begin{align}\label{eq1}
\int_X \int_X (P_t f (x) -P_t f (y))^2 d\nu_r (x,y) \le  C \frac{r^{d_H+d_W}}{t}  \|f\|^2_{L^2(X,\mu)}.
\end{align}

We now consider the case where $f \in L^\infty(X,\mu)$. From the weak Bakry-\'Emery estimate we have for  $\nu_r$ almost every $x,y \in X$
\begin{align}\label{eq2}
|P_tf(x)-P_tf(y)|  \le C \frac{r^\kappa}{t^{\kappa/d_W}}  \| f \|_{L^\infty(X,\mu)}.
\end{align}

Consider the map $\mathcal{P}_t$ defined by $\mathcal{P}_t f (x,y)= P_t f (x) -P_t f (y)$.  The estimate \eqref{eq1} says that $\mathcal{P}_t: L^2(X,\mu)\to L^2(X \times X, \nu_r)$ is bounded by  $ C \frac{r^{(d_H+d_W)/2}}{\sqrt{t}} $. The estimate \eqref{eq2} says that $\mathcal{P}_t: L^\infty(X,\mu)\to L^\infty(X \times X, \nu_r)$ is bounded by $ C \frac{r^\kappa}{t^{\kappa/d_W}}$. From the Riesz-Thorin interpolation theorem, we deduce that if $2 \le p <+\infty$, $\mathcal{P}_t: L^p(X,\mu)\to L^p(X \times X, \nu_r)$ is bounded with 
\[
\| \mathcal{P}_t \|_{L^p(X,\mu)\to L^p(X \times X, \nu_r)}\le C \frac{r^{ \kappa \left(1-\frac{2}{p}\right)+\frac{1}{p} (d_H+d_W)}}{ t^{\beta_p}}.
\]
This can be rewritten as
\[
\int_X \int_{B(x,r)} |P_t f (x) -P_t f (y)|^p d\mu (y) d\mu(x)  \le C \frac{r^{ \kappa\left(p-2\right)+d_H+d_W}}{ t^{p\beta_p}} \|f \|^p_{L^p(X,\mu)}.
\]
We deduce that
\[
\sup_{r >0} \frac{1}{r^{ \kappa\left(p-2\right)+d_H+d_W}} \int_X \int_{B(x,r)} |P_t f (x) -P_t f (y)|^p d\mu (y) d\mu(x)  \le  \frac{C}{ t^{p\beta_p}} \|f \|^p_{L^p(X,\mu)}
\]
and conclude from Theorem \ref{Besov characterization}.
\end{proof}

Among our main applications of the weak Bakry-\'Emery condition $wBE(\kappa)$ are the following pseudo-Poincar\'e inequalities for the semigroup when $1\leq p\leq2$.  These complement those for $p\geq2$ which were obtained without $wBE(\kappa)$ in~\cite[Proposition~5.3]{ABCRST1}.

\begin{proposition}[Pseudo-Poincar\'e inequalities]\label{pseudo poincare}
Let $1\leq p\leq 2$.  If $(X,\mu,\mathcal{E},\mathcal{F})$ satisfies the weak Bakry-\'Emery condition $wBE(\kappa)$, then there is $C>0$ such that for every $f \in L^p(X,\mu)$, and $t \ge 0$,
\begin{equation*}
\| P_t f -f \|_{L^p(X,\mu)} \le C t^{\beta_p} \liminf_{\tau \to 0^+}\frac{1}{ \tau^{\beta_p}}\biggl( \int_X\int_X p_{\tau}(x,y) |f(x)-f(y)|^p \,d\mu(x)\,d\mu(y)\biggr)^{1/p}.
\end{equation*}
\end{proposition}

\begin{proof}
We denote
\[
\mathcal{E}_\tau(u,v)= \frac1\tau\int_X u(I-P_\tau)v\,d\mu 
 = \frac1{2\tau} \int_X\int_X p_\tau(x,y)\bigl( u(x)-u(y)\bigr)\bigl(v(x)-v(y) \bigr) d\mu(x)d\mu(y).
 \]
Fix $f\in L^p(X,\mu)$ and $h\in L^q(X,\mu)$ where $p$ and $q$ are conjugate exponents.   Using the strong continuity of the semigroup $P_t$ in $L^1(X,\mu)$, one has for $t \ge 0$,
\begin{align}
\int_X (P_t f -f ) h\, d\mu
&= \lim_{\tau\to0^+} \frac{1}{\tau} \int_X \Bigl( \int_t^{t+\tau} P_s f ds - \int_0^{\tau} P_s f ds \Bigr) h\, d\mu \notag\\
&= \lim_{\tau \to 0^+} \frac{1}{ \tau} \int_0^t\int_X (P_{s+\tau} f -P_{s} f) h \, d\mu\, ds 
= - \lim_{\tau \to 0^+} \int_0^t \mathcal{E}_\tau (P_s f ,h) ds. \label{eq:Pt-IbyVar1}
\end{align}
In the case $p=1$ we now estimate $\mathcal{E}_\tau (P_s f ,h) =\mathcal{E}_\tau (f ,P_s h)$ using $wBE(\kappa)$ and then~\eqref{eq:boundpolyptbypct} as follows:
\begin{align}\label{eq:ppoincareeqn1}
2\big|\mathcal{E}_\tau(f,P_sh)\bigr|
 &\le \frac{1}{\tau}\int_X \int_X p_\tau (x,y) | P_s h(x)-P_sh(y)||f(x)-f(y)| d\mu(x) d\mu(y) \notag\\
 & \le C\| h \|_{L^\infty}\frac1{\tau s^{\kappa/d_W}} \int_X \int_X p_\tau (x,y) d(x,y)^\kappa |f(x)-f(y)| d\mu(x) d\mu(y) \notag\\
 &\leq C\| h \|_{L^\infty} \frac1 {s^{\kappa/d_W} \tau^{\beta_1}} \int_X \int_X p_{c\tau} (x,y) |f(x)-f(y)| d\mu(x) d\mu(y).
\end{align}
The $1<p\leq2$ case is similar except that we use H\"older's inequality and Theorem~\ref{P:PtinBesovp4} to obtain
\begin{align}\label{eq:ppoincareeqn2}
2\big|\mathcal{E}_\tau(f,P_sh)\bigr|
&\leq \frac{1}{\tau}\int_X \int_X p_\tau (x,y) | P_s h(x)-P_sh(y)||f(x)-f(y)| d\mu(x) d\mu(y) \notag \\
&\leq \tau^{-\beta_q} \bigl\| \mathcal{P}_s h \|_{L^q(X\times X,p_\tau d\mu\otimes d\mu)} 
 \tau^{-\beta_p} \bigl\| \mathcal{P}_0 f \|_{L^p(X\times X,p_\tau d\mu\otimes d\mu)} \notag \\
&\leq C s^{-\beta_q} \|h\|_{L^q(X,\mu)} \tau^{-\beta_p} \biggl(\int_X\int_X p_\tau(x,y)|f(x)-f(y)|^p\,d\mu(x)\,d\mu(y)\biggr)^{1/p}.
\end{align}
Integrating~\eqref{eq:ppoincareeqn1} and~\eqref{eq:ppoincareeqn2} over $s\in(0,t)$ as in~\eqref{eq:Pt-IbyVar1} and taking $\liminf_{\tau\to0^+}$ give for $1\leq p\leq 2$:
\begin{align*}
\biggl| \int_X (P_t f -f ) h d\mu \biggr| \le C \frac{t^{\beta_p}}{\beta_p}  \| h \|_{L^q(X,\mu)} \liminf_{\tau \to 0^+}\frac{1}{ \tau^{\beta_p}}\biggl(\int_X\int_X p_{\tau}(x,y) |f(x)-f(y)|^p \,d\mu(x)\,d\mu(y)\biggr)^{1/p},
\end{align*}
so the conclusion  follows by $L^p-L^q$ duality.
\end{proof}

\subsection{$L^p$--Besov critical exponents and generalized Riesz transforms}\label{S:critical_exp_local}
Recall that the density or triviality of the spaces $\B^{p,\alpha}(X)$ can be described using the critical exponents $\alpha_p^*\leq\alpha_p^\#$, where the former is supremal for density in $L^p$ and the latter supremal for containing non-constant functions, see~\eqref{eqn:defnBesovcritexponents}.  In this section we give bounds on the exponents that follow from the weak Bakry-\'Emery condition~\eqref{WBECD}, using the notation $\beta_p$ from~\eqref{eq:defnofbetap}. Bounds without this condition were proved in~\cite[Section~5.2]{ABCRST1}, namely $\frac{1}{2} \le \alpha^*_p $ if $1 \le p \le 2$ and $\alpha_p^\#\leq \frac{1}{2}$ if $ p \ge 2$.

\begin{theorem}\label{estimate kappa alpha}
If $(X,\mu,\mathcal{E},\mathcal{F})$ satisfies the weak Bakry-\'Emery condition $wBE(\kappa)$ then:
\begin{itemize}
\item For $1 \le p \le 2$ we have $\frac{1}{2} \le \alpha^*_p \le \alpha_p^\#\leq \beta_p$. 
\item For $ p \ge 2$ we have $\beta_p \le \alpha^*_p \le\alpha_p^\#\leq \frac{1}{2}$.
\end{itemize}
\end{theorem}
\begin{proof}
If $1\leq p\leq 2$ and $\alpha>\beta_p$ then for $f \in \B^{p,\alpha}(X)$ we have
\begin{equation*}
	\liminf_{\tau \to 0^+}\frac{1}{\tau^{\beta_p}} \biggl( \int_X\int_Xp_\tau(x,y) |f(x)-f(y)|^p \,d\mu(x)\,d\mu(y)\biggr)^{1/p}
	\leq \liminf_{\tau \to 0^+} \tau^{\alpha-\beta_p}\|f\|_{p,\alpha} = 0,
	\end{equation*}
so Proposition~\ref{pseudo poincare} implies $P_tf=f$ for all $t>0$. Since the heat kernel upper bound implies by interpolation the estimate
\[
\| P_t f \|_{L^\infty(X,\mu)} \le \frac{C}{t^{\frac{d_H}{pd_W}}} \| f \|_{L^p(X,\mu)},
\]
we deduce by letting $t \to +\infty$ that $f$ is zero and thus constant.  This gives $\alpha_p^\#\leq\beta_p$ for $1\leq p\leq2$.  The bound $\alpha_p^*\geq\frac12$ is true even without the weak Bakry-\'Emery inequality, see~\cite[Proposition~5.6]{ABCRST1}.

Since $\|f\|_{2,1/2}^2=2\DF(f,f)$, see~\cite[Proposition~4.6]{ABCRST1}, Proposition~\ref{pseudo poincare} also shows that $\DF(f,f)=0$ implies $f$ is constant, under which hypothesis $\alpha_p^\#\leq \frac{1}{2}$ was proven for $p\geq2$ in~\cite[Proposition~5.6]{ABCRST1}. For the other inequality, if $p \ge 2$ then Theorem \ref{P:PtinBesovp4} says $P_tf \in \B^{p,\beta_p}(X)$ whenever $f \in L^p(X,\mu)$; since $P_t f \to f$ in $L^p(X,\mu)$ when $t \to 0$ we conclude that $\B^{p,\beta_p}(X)$ is dense in $L^p(X,\mu)$ and hence $\beta_p\leq\alpha_p^*$.  
\end{proof}

\begin{remark}
Note that from $\beta_1 \ge \frac{1}{2}$ one deduces that $\kappa \le \frac{d_W}{2}$. If $(X,d)$ satisfies a chain condition, then from Remark \ref{chain condition}, one has $\kappa \le 1$. Thus, in that case, if $\kappa =\frac{d_W}{2}$, one actually has $\kappa=1$ and $d_W=2$, meaning that the heat kernel has Gaussian  estimates.
\end{remark}

The next corollary shows that the upper bound $\kappa=\frac{d_W}{2}$ may only be achieved in Dirichlet spaces that admit a carr\'e du champ. 

\begin{corollary}\label{carre du champ}
If $(X,d,\mu)$ satisfies $wBE(\kappa)$ with $\kappa=\frac{d_W}{2}$, then the form $\mathcal{E}$ admits a carr\'e du champ operator. More precisely, for every $f\in L^\infty (X,\mu)\cap\mathcal{F}$  there is $\Gamma(f,f)\in L^1(X,\mu)$ such that for all $g\in L^\infty(X,\mu)\cap\mathcal{F}$, 
\begin{equation}\label{eqn-ac}
\int_X g \Gamma (f,f) d\mu =2 \mathcal{E}(gf,f)-\mathcal{E}(f^2,g).
\end{equation}
\end{corollary}
\begin{proof}
If $(X,d,\mu)$ satisfies $wBE(\kappa)$ with $\kappa=\frac{d_W}{2}$ then $\beta_p=\frac12$ for every $p \ge 1$,  so  Theorem~\ref{estimate kappa alpha} yields $\alpha^*_p=\alpha_p^\#=\frac12$. It follows that $\mathbf{B}^{p,1/2}(X)\cap\mathcal{F}$ is dense in $\mathcal{F}$ in the norm $\bigl(\DF(f,f)+\|f\|_{L^2}^2\bigr)^{1/2}$, and the existence of the carr\'e du champ operator is provided by Corollary~4.12 in~\cite{ABCRST1}.
\end{proof}

It is natural to ask under what conditions one has $\alpha_p^*=\alpha_p^\#=\beta_p$ for the optimal $\kappa$ such that 
$wBE(\kappa)$ is satisfied.
This is a question for future study, however it seems worthwhile to  briefly comment here on the connection of this problem to a natural notion of Riesz transform.  For $p>1$, $\alpha \in (0,1]$, let us say that $X$ satisfies $(R_{p,\alpha})$ if there exists a constant $C=C_{p,\alpha}$ such that 
\begin{equation}
\| f \|_{p,\alpha} \le C \| (-L)^{\alpha} f \|_{L^p(X,\mu)},  \tag{$R_{p,\alpha}$}
\end{equation} 
for all $f$ in a suitable domain. For instance, in the strictly local framework of~\cite{ABCRST2}, under the strong Bakry-\'Emery curvature condition, Corollary~4.10 in~\cite{ABCRST2} gives $\| f \|^p_{p,1/2} \sim \int_X \Gamma(f,f)^{p/2} d\mu$ and $(R_{p,1/2})$ is therefore equivalent to boundedness of the Riesz transform in $L^p(X,\mu)$.

In the present setting one has the following result.
\begin{lemma}
For $1<p \le 2$, validity of $(R_{p,\beta_p})$ implies $\alpha^*_p =\beta_p$.
\end{lemma}
\begin{proof}
By analyticity of the semigroup, one has for every $f \in L^p(X,\mu)$, 
\[
\| P_t f \|_{p,\beta_p} \le C \| (-L)^{\beta_p} P_tf \|_{L^p(X,\mu)} \le \frac{C}{t^{\beta_p}} \| f \|_{L^p(X,\mu)}.
\]
Therefore, for $t>0$, $P_t: L^p(X,\mu) \to \B^{p,\beta_p}(X)$ is bounded. 
In addition, the heat semigroup is strongly continuous on $L^p(X,\mu)$ for $1<p\le 2$, i.e., $\|P_tf-f\|_{L^p(X,\mu)}\to 0$ as $t$ tends to 0. We thus conclude that $\B^{p,\beta_p}(X)$ is dense in $L^p(X,\mu)$ and $\alpha_p^*\ge \beta_p$. The proof is complete by recalling that $\alpha_p^*\le \beta_p$ from Theorem \ref{estimate kappa alpha}. 
\end{proof}

\section{BV class}\label{section BV}
We now turn to the highlight of our paper and define the BV class. As always, we assume that $(X,d,\mu)$ satisfies the standing Assumption \ref{A1} and Assumption~\ref{sGKHE}.

\subsection{$L^1$ Korevaar-Schoen and BV class}

To motivate our definition of the BV class, we first introduce the scale of $L^1$ Korevaar-Schoen type spaces (see \cite{KS93,KST04} for the classical definitions). For $\lambda >0$, we define the space $KS^{\lambda,1}(X)$ to be those $f\in L^1(X,\mu)$ for which 
\[
\Vert f\Vert_{KS^{\lambda,1}(X)}:=\limsup_{r\to 0^+}\iint_{\Delta_r}\frac{|f(y)-f(x)|}{r^\lambda \mu(B(x,r))}\, d\mu(y)\, d\mu(x)<+\infty,
\]
where we recall that $\Delta_r=\{(x,y)\in X\times X:d(x,y)<r\}$.
Similarly, we define the space $\mathcal{KS}^{\lambda,1}(X)$ to be those $f\in L^1(X,\mu)$ for which
\[
\Vert f\Vert_{\mathcal{KS}^{\lambda,1}(X)}:=\sup_{r>0}\iint_{\Delta_r}\frac{|f(y)-f(x)|}{r^\lambda\mu(B(x,r))}\, d\mu(y)\, d\mu(x)<+\infty.
\]

\begin{proposition}\label{KS-BEsov}
For $\lambda >0$,
\[
KS^{\lambda,1}(X)=\mathcal{KS}^{\lambda,1}(X)=\mathbf{B}^{1,\frac{\lambda}{d_W}}(X).
\]
Moreover, $\mathcal{KS}^{\lambda,1}(X)$ and $\mathbf{B}^{1,\frac{\lambda}{d_W}}(X)$ have equivalent norms.
\end{proposition}

\begin{remark}
We note that $KS^{\lambda,1}(X)$ and $\mathcal{KS}^{\lambda,1}(X)$ do not have equivalent norms in general, but that the inequality $\Vert f\Vert_{KS^{\lambda,1}(X)} \le \Vert f\Vert_{\mathcal{KS}^{\lambda,1}(X)}$ is always true. 
\end{remark}

\begin{proof}
The fact that $\mathcal{KS}^{\lambda,1}(X)$ and $\mathbf{B}^{1,\frac{\lambda}{d_W}}(X)$ are equal with comparable norms follows from Theorem~\ref{Besov characterization}, and the inclusion $ \mathcal{KS}^{\lambda,1}(X) \subset KS^{\lambda,1}(X)$ is immediate. Thus it remains to prove that $KS^{\lambda,1}(X) \subset \mathcal{KS}^{\lambda,1}(X)$. If $f\in KS^{\lambda,1}(X)$, take $\varepsilon>0$ so that
we have 
\[
\sup_{r\in(0,\varepsilon]} \iint_{\Delta_r} \frac{|f(y)-f(x)|}{r^\lambda \mu(B(x,r))}\, d\mu(y)\, d\mu(x) \le 2\|f\|_{KS^{\lambda,1}}.
\]
Using $|f(x)-f(y)|\leq |f(x)|+|f(y)|$ and
$c_{1}r^{d_{H}} \le \mu\bigl(B(x,r)\bigr)\leq c_{2}r^{d_{H}}$, one has for $r \ge \ve$
\begin{align*}
\iint_{\Delta_r}\frac{|f(y)-f(x)|}{r^\lambda \mu(B(x,r))}\, d\mu(y)\, d\mu(x)
&\leq\frac{C}{r^{\lambda+d_{H}}}\iint_{\Delta_r}\bigl(|f(x)|+|f(y)|\bigr)\,d\mu(x)\,d\mu(y)\\
&\leq \frac{C}{r^{\lambda+d_{H}}}\int_X|f(y)|\mu\bigl(B(y,r)\bigr)\,d\mu(y)\\
 &\le \frac{C}{r^{\lambda}}\|f\|_{L^{1}(X,\mu)}\leq\frac{C}{\varepsilon^{\lambda}}\|f\|_{L^{1}(X,\mu)}.
\end{align*}
Combining this with the estimate when $r\in(0,\varepsilon]$ gives $\|f\|_{ \mathcal{KS}^{\lambda,1}}\leq C \max \bigl\{\varepsilon^{-\lambda}\|f\|_{L^{1}(X)}, 2\|f\|_{KS^{\lambda,1}}\bigr\}$.
\end{proof}

In this, and what follows, we set
\begin{equation}\label{e:lambda-critical}
\begin{aligned}
	\lambda_1^*&:=\sup \{ \lambda >0\, :\, KS^{\lambda,1}(X) \text{ is dense in } L^1(X,\mu)\},\\
	\lambda_1^\#&:=\sup \{ \lambda >0\, :\, KS^{\lambda,1}(X) \text{ contains non-constant functions.}\}.
	\end{aligned}
	\end{equation}

\begin{remark}\label{Remark KS Dirichlet}
Though it is not directly relevant to the $L^1$ theory of BV functions developed here, we point out it is possible to prove that for the $L^2$ critical exponents one has $\lambda_2^*=\lambda_2^\#=\frac{d_W}{2}$ and that $KS^{\frac{d_W}{2},2}(X)=\mathcal{KS}^{\frac{d_W}{2},2}(X)=\mathcal F$ with comparable norms. The proof will appear in a later work.
\end{remark}

We can now define the BV class.  
\begin{definition}\label{def-BV}
Let $BV(X):=KS^{\lambda^\#_1,1}(X)$and for $f \in BV(X)$ define
\[
\mathbf{Var} (f):=\liminf_{r\to 0^+}\iint_{\Delta_r}\frac{|f(y)-f(x)|}{r^{\lambda_1^\#} \mu(B(x,r))}\, d\mu(y)\, d\mu(x).
\]
\end{definition}
Note that from Proposition \ref{KS-BEsov} one has $BV(X)=\mathbf{B}^{1,\alpha_1^\#}(X)$, where $\alpha_1^\#=\frac{\lambda_1^\#}{d_W}$ is the $L^1$--Besov critical exponent.

\begin{example}
For nested fractals $\lambda_1^\#=\lambda_1^*=d_W\alpha_1^*=d_H$, as we will see in Theorem~\ref{T:nested}.  For the Sierpinski carpet one conjectures $\lambda_1^\#=\lambda_1^*=d_H-d_{tH}+1$, where $d_{tH}=\dfrac{\ln 2}{\ln 3}+1$ is the topological-Hausdorff dimension  of the Sierpinski carpet defined in \cite[Theorem 5.4]{BBE15}, see Conjecture~\ref{conj-SC}.
\end{example}

\subsection{Locality property of BV functions}\label{S:BVfunctions}

Throughout the remainder of this section we make the following crucial assumption. 
\begin{assumption}\label{assumpt:wBEatcrit}
$X$ satisfies $wBE(\kappa)$ with  $\kappa=d_W-\lambda_1^\#=d_W(1-\alpha_1^\#).$
\end{assumption}

From Lemma~\ref{non-unique kappa} and Theorem~\ref{estimate kappa alpha}, this implies that $$\kappa=\sup \{ \kappa':  X \text{ satisfies } wBE(\kappa')  \}.$$

\begin{example}\label{example-nested}
From Theorems \ref{thm-xBE-FD} and \ref{tensorization}, if $X$ is a nested fractal, then for every $n \ge 1$, $X^n$ satisfies $wBE(\kappa)$ with $\kappa=d_W(1-\alpha_1^*)=d_W-\dim_H(X)$.
\end{example}

\begin{remark}\label{rem-wBE*}
If $X$ satisfies $wBE(\kappa)$ with  $\kappa=d_W-\lambda_1^*=d_W(1-\alpha_1^*)$, then from Theorem \ref{estimate kappa alpha}, one has $\lambda_1^*=\lambda_1^\#$ and thus Assumption \ref{assumpt:wBEatcrit} is satisfied.
\end{remark}

 For $f \in L^1(X,\mu)$, let
\begin{equation}
\mathcal{M}_r f (x)=\frac{1}{r^{d_W-\kappa}\mu(B(x,r))} \int_{B(x,r)} |f(x) -f(y)| d\mu(y) 
\end{equation}
so that
\begin{equation}\label{eq:VarintermsofMf}
\mathbf{Var} (f)=\liminf_{r\to 0^+}\int_X\mathcal{M}_r f (x)\, d\mu(x).
\end{equation}

An important result is that $wBE(\kappa)$ implies the following locality property of the $\|\cdot\|_{1,1-\frac\kappa{d_W}}$  Besov seminorm:

\begin{theorem}\label{bounded embedding}
There exist constants $c,C>0$ such that for every $f \in BV(X)$,
\[
c \mathbf{Var} (f) \le \| f \|_{1,1-\frac{\kappa}{d_W}} \le C \mathbf{Var} (f),
\]
where we recall that $\|\cdot\|_{1,1-\frac\kappa{d_W}}$ is the Besov seminorm for $\B^{1,1-\kappa/d_W}(X)$, see~\eqref{eq:defnofheatbesovnorm}. Consequently there is $C>0$ such that for every $f \in BV(X)$,
\[
\sup_{r >0} \int_X \mathcal{M}_r f (y) d\mu(y) \le C\liminf_{r \to 0^+} \int_X \mathcal{M}_r f (y) d\mu(y),
\]
and the two Korevaar-Schoen spaces $KS^{\lambda,1}(X)$ and $\mathcal{KS}^{\lambda,1}(X)$ have equivalent norms for $\lambda=d_W-\kappa$.
\end{theorem}

We divide the proof of Theorem \ref{bounded embedding} into several lemmas.   For $f \in L^1(X,\mu)$, we will use the notation
\begin{gather}
\mathbf{Var}_* (f)
 =\liminf_{t \to 0^+}\frac{1}{ t^{1-\frac{\kappa}{d_W}}}\int_X\int_X p_{t}(x,y) |f(x)-f(y)| \,d\mu(x)\,d\mu(y).
  \end{gather}

The next two results are useful consequences of the Fubini theorem. In this and several later results, for $f\in L^1(X,\mu)$ we use the notation $E_t(f)=\{x\in X: f(x)>t\}$. We note that if $f \ge 0$ a.e. and $t>0$, then $\mathbf{1}_{E_t(f)} \in L^1(X,\mu)$ and therefore Proposition \ref{pseudo poincare} applies to $\mathbf{1}_{E_t(f)}$.

\begin{lemma}\label{lem:Fubini for BV estimate}
If $f\in L^1(X,\mu)$ and $g \in L^\infty (X,\mu)$, $f,g \ge 0$ a.e., then
\begin{multline*}
\int_X\int_X g(y)p_s(x,y)|f(x)-f(y)|\, d\mu(x)d\mu(y)
\\ \leq \int_0^{\infty} \biggl( \bigl\|P_s(g\mathbf{1}_{E_t(f)})-g\mathbf{1}_{E_t(f)}\bigr\|_{L^1(X,\mu)}+ \bigl\|P_s(g\mathbf{1}_{X\setminus E_t(f)})-g\mathbf{1}_{X\setminus E_t(f)}\bigr\|_{L^1(X,\mu)}\biggl)\, dt.
\end{multline*}
In particular, if $g\equiv1$ then
\[
\int_X\int_X p_s(x,y)|f(x)-f(y)|\, d\mu(x)d\mu(y)
\leq 2 \int_0^{\infty} \bigl\|P_s(\mathbf{1}_{E_t(f)})-\mathbf{1}_{E_t(f)}\bigr\|_{L^1(X,\mu)}\, dt.
\]
\end{lemma}
\begin{proof}
Let $A=\{(x,y)\in X\times X\, :\, f(x)<f(y)\}$. Then for $s>0$,
\begin{align*}
 \lefteqn{ \int_X\int_X g(y) p_s(x,y) |f(x)-f(y)|\, d\mu(x)d\mu(y)}\quad& \\
 &=\int_A (g(x)+g(y))p_s(x,y)|f(x)-f(y)|\, d\mu(x)d\mu(y)\\
 &=\int_A \int_{f(x)}^{f(y)}dt\, p_s(x,y) (g(x)+g(y))d\mu(x)d\mu(y)\\
 &=\int_X\int_X \int_0^{\infty} \mathbf{1}_{[f(x),f(y))}(t) \, dt \mathbf{1}_A(x,y)\, p_s(x,y) (g(x)+g(y))d\mu(x)d\mu(y) \\
 &=\int_0^{\infty} \int_X\int_X \mathbf{1}_{E_t(f)}(y)[1-\mathbf{1}_{E_t(f)}(x)]\, p_s(x,y)\, (g(x)+g(y))d\mu(x)d\mu(y)\, dt\\
 &=\int_0^\infty \int_{X\setminus E_t(f)}P_s(g\mathbf{1}_{E_t(f)})(x)\, d\mu(x)\, dt+\int_0^\infty \int_{ E_t(f)}P_s(g\mathbf{1}_{X\setminus E_t(f)})(x)\, d\mu(x)\, dt. 
\end{align*}
The result now follows from the fact that $|P_s(g\mathbf{1}_{E_t(f)})-g\mathbf{1}_{E_t(f)}|=|P_s(g\mathbf{1}_{E_t(f)})|$ on $X\setminus E_t(f)$ and similarly for the other term.

In particular, when $g\equiv1$, there holds $\int_{X\setminus E_t(f)}P_s(\mathbf{1}_{E_t(f)})(x)\, d\mu(x)=\int_{ E_t(f)}P_s(\mathbf{1}_{X\setminus E_t(f)})(x)\, d\mu(x)$. Thus we conclude the proof.
\end{proof}

\begin{lemma}\label{lem:intVar*byVar*}
If $f\in L^1(X,\mu)$, $f \ge 0$ a.e., then
\begin{equation*}
 \int_0^ {\infty} \mathbf{Var}_*(\mathbf{1}_{E_t(f)})\,dt \leq \mathbf{Var}_* (f).
 \end{equation*}
\end{lemma}
\begin{proof}
We compute by using Lemma \ref{lem:Fubini for BV estimate} and Fatou's lemma 
\begin{align*}
 \lefteqn{\int_0^\infty \mathbf{Var}_*(\mathbf{1}_{E_t(f)})\,dt} \quad&\\
 &= \int_0^\infty \liminf_{\tau \to 0^+}\frac{1}{ \tau^{1-\frac{\kappa}{d_W}}}\int_X\int_X p_{\tau }(x,y) |\mathbf{1}_{E_t(f)}(x)-\mathbf{1}_{E_t(f)}(y)| \,d\mu(x)\,d\mu(y) \, dt\\
 &\leq \liminf_{\tau \to 0^+}\frac{2}{ \tau^{1-\frac{\kappa}{d_W}}} \int_A p_{\tau }(x,y) \biggl( \int_0^{\infty} |\mathbf{1}_{E_t(f)}(x)-\mathbf{1}_{E_t(f)}(y)| dt \biggr) \,d\mu(x)\,d\mu(y) \\
 &\leq \liminf_{\tau \to 0^+}\frac{2}{ \tau^{1-\frac{\kappa}{d_W}}} \int_A p_{\tau }(x,y) (f(y)-f(x)) \,d\mu(x)\,d\mu(y) \\
 &= \liminf_{\tau \to 0^+}\frac{1}{ \tau^{1-\frac{\kappa}{d_W}}} \int_X \int_X p_{\tau }(x,y) |f(y)-f(x)| \,d\mu(x)\,d\mu(y).\qedhere
\end{align*}

\end{proof}
Combining our results thus far we have the following.
\begin{lemma}\label{Lemma 2:BV embed}
There is $C>0$ so for every $f \in L^1(X,\mu)$,
\[
 \| f \|_{1,1-\frac{\kappa}{d_W}} \le C \mathbf{Var}_* (f).
\]
\end{lemma}
\begin{proof}
We first recall the $L^1$-pseudo-Poincar\'e inequality in Proposition \ref{pseudo poincare}, namely, for any $f\in L^1(X,\mu)$ and $t>0$,
\[
\| P_t f -f \|_{L^1(X,\mu)} \le C t^{1-\kappa/d_W}  \mathbf{Var}_*(f).
\]
Considering $f_n=(f+n)_{+}$ and letting $E_t(f_n)=\{x\in X: f_n(x)>t\}$, we apply the above pseudo-Poincar\'e inequality to the integrand in the expression in Lemma~\ref{lem:Fubini for BV estimate} with $g\equiv1$ and obtain from Lemma~\ref{lem:intVar*byVar*}
\begin{equation}\label{eq:Lemma2BVembedeq1} 
\begin{split}
\int_X\int_X p_s(x,y)|f_n(x)-f_n(y)|\, d\mu(x)d\mu(y)
&\le 2 \int_0^{\infty} \| P_s \mathbf 1_{E_t(f_n)} - \mathbf 1_{E_t(f_n)} \|_{L^1(X,\mu)} dt
 \\&\leq 
 C s^{1-\kappa/d_W} \int_{0}^{\infty} \mathbf{Var}_*(\mathbf{1}_{E_t(f_n)}) dt
\\ &\leq C s^{1-\kappa/d_W} \mathbf{Var}_*(f_n). 
 \end{split}
 \end{equation}
Letting $n \to \infty$, dividing by $s^{1-\kappa/d_W}$ and taking the supremum concludes the proof.
\end{proof}
%
%

Combining the preceeding results with arguments from Theorem~\ref{Besov characterization} proves our final lemma and completes the proof of Theorem~\ref{bounded embedding}.

\begin{lemma}\label{comparison var}
There exist constants $C_1,C_2>0$ such that for every $f \in BV(X)$,
\[
 C_1 \mathbf{Var}_* (f) \le \mathbf{Var}(f) \le C_2 \mathbf{Var}_* (f).
\]
\end{lemma}

\begin{proof}
The second inequality is obtained using the definition and Theorem~\ref{Besov characterization} to see that 
\[
\mathbf{Var}(f) \le \sup_{r>0} \int_X \mathcal{M}_r f (y) d\mu(y) \le C \| f \|_{1,1-\frac{\kappa}{d_W}}
\]
and applying Lemma~\ref{Lemma 2:BV embed}. To prove the first inequality we decompose as in the proof of Theorem~\ref{Besov characterization}. Fix $ f \in L^1(X,\mu)$, and $\delta,t>0$.  Write
\begin{equation*}
\Psi (t)=\frac{1}{ t^{1-\frac{\kappa}{d_W}}}\int_X\int_X p_{t}(x,y) |f(x)-f(y)| \,d\mu(x)\,d\mu(y)
\end{equation*}
and estimate as follows. Let $r=\delta t^{1/d_W}$. For $d(x,y)<\delta t^{1/d_W}$ the sub-Gaussian upper bound~\eqref{eq:subGauss-upper} implies $p_t(x,y)\leq C t^{-d_H/d_W}$, so that
\begin{align*}
\frac{1}{ t^{1-\frac{\kappa}{d_W}}}\iint_{\Delta_r} p_{t}(x,y) |f(x)-f(y)| \,d\mu(x)\,d\mu(y)
	&\leq \frac{C}{t^{(d_W+d_{H}-\kappa)/d_{W}}}\iint_{\Delta_r} |f(x)-f(y)| \,d\mu(x)\,d\mu(y) \\
	&\leq C\delta^{d_W+d_H-\kappa} \int_X \mathcal{M}_{\delta t^{1/d_W}} f (y) d\mu(y)
	:=\Phi(t).
	\end{align*}
For $d(x,y)>\delta t^{1/d_W}$ we instead use the sub-Gaussian bounds~\eqref{eq:subGauss-upper} to see there are $c,C>1$ and $c'>0$ such that
\begin{equation*}
	p_t(x,y)\leq C\exp\biggl( -c'\Bigl(\frac{d(x,y)^{d_W}}t\Bigr)^{\frac1{d_W-1}}\biggr) p_{ct}(x,y)
	\leq C\exp\bigl(-c' \delta^{\frac{d_W}{d_W-1}}\bigr) p_{ct}(x,y)
	\end{equation*}
and therefore
\begin{align}
	\Psi(t)
	&\leq \Phi(t) + \frac{1}{ t^{1-\frac{\kappa}{d_W}}}\int_X\int_{X\setminus B(y,r)} p_{t}(x,y) |f(x)-f(y)| \,d\mu(x)\,d\mu(y)\notag\\
	&\leq \Phi(t)+ C\exp\bigl(-c' \delta^{\frac{d_W}{d_W-1}}\bigr) \int_X\int_{X\setminus B(y,r)} p_{ct}(x,y) |f(x)-f(y)| \,d\mu(x)\,d\mu(y)\notag\\
	&\le \Phi(t) + A \Psi(ct), \label{eq:ineqPhiandPsi}
	\end{align}
where $A$ is a constant that can be made as small as we desire by making $\delta$ large enough. We choose $\delta$ so that $A<\frac12$. Observe that  with $t=c^{-(n+1)}$ ~\eqref{eq:ineqPhiandPsi} gives
\begin{equation}\label{eq:iteratedboundforPsi}
	\Psi(c^{-(n+1)})\leq \Phi(c^{-(n+1)})+A\Psi(c^{-n}).
	\end{equation}

Now suppose $f\in BV(X)$. Then $\limsup_{r \to 0^+} \int_X \mathcal{M}_r f (y) d\mu(y)<\infty$ and therefore there is $M'$ so $\sup_{n \ge1} \Phi \left(c^{-n}\right) \le M' <+\infty$.  For $M=\max\{M',\Psi(c^{-1})\}$ it is easily checked by induction from~\eqref{eq:iteratedboundforPsi} that $\Psi(c^{-n})\leq \frac{M}{1-A}(1-A^n)$.  This implies
$\mathbf{Var}_*(f)=\liminf_{t\to0^+}\Psi(t)<\infty$. On the other hand, from Lemma \ref{Lemma 2:BV embed} one has 
$\Psi(t)\le C \mathbf{Var}_*(f)$. Hence the estimate in \eqref{eq:ineqPhiandPsi} yields
\[
\Psi(t)\le \Phi(t)+AC \mathbf{Var}_*(f).
\]
Fixing $\delta$ such that $AC<1$ and taking $\liminf_{t\to 0^+}$, one deduces $C_1 \mathbf{Var}_* (f) \le \mathbf{Var}(f)$.

\end{proof}

\subsection{Co-area formula}\label{S:Co-area formula}

From computations of the previous section, we may immediately deduce several other properties of $\mathbf{Var}$.

\begin{lemma} If $f\in L^1(X,\mu)$ has $\mathbf{Var} (f)=0$ then $f=0$. Moreover there is $C>0$ so for $f,g \in BV(X)$
\begin{gather*}
 \mathbf{Var}(f+g) \le C ( \mathbf{Var}(f)+\mathbf{Var}(g)),\\
 \mathbf{Var}(fg) \le C (\|f \|_{L^\infty(X,\mu)} \mathbf{Var}(g)+\|g \|_{L^\infty(X,\mu)} \mathbf{Var}(f)).
 \end{gather*}
\end{lemma}
\begin{proof}
Let $f\in L^1(X,\mu)$ such that $\mathbf{Var} (f)=0$. By Theorem \ref{bounded embedding}, one has $\| f \|_{1,1-\kappa/d_W}=0$ and thus $f$ is constant. We conclude the first point by recalling that $f \in L^1(X,\mu)$. In light of Theorem \ref{bounded embedding} it suffices to establish the stated inequalities for $\| \cdot \|_{1,1-\frac{\kappa}{d_W}}$, but we recall
\begin{equation*}
 \| h\|_{1,1-\frac{\kappa}{d_W}}=\sup_{t >0}\frac1{ t^{1-\frac{\kappa}{d_W}}}\int_X\int_X p_{t}(x,y) |h(x)-h(y)| \,d\mu(x)\,d\mu(y)
 \end{equation*}
and when $h=f+g$ write $|h(x)-h(y)|\leq|f(x)-f(y)|+|g(x)-g(y)|$ while for $h=fg$ write
\begin{equation*}
 |h(x)-h(y)|\leq |f(x)||g(x)-g(y)|+|g(y)||f(x)-f(y)|
 \end{equation*}
from which the result is immediate.
\end{proof}

The next result plays the role of a co-area formula. 

\begin{theorem}[Co-area Formula]\label{coarea formula 1a}
For $f \in L^1(X,\mu)$, $f \ge 0$ a.e., let $E_t(f)=\{x\in X\, :\, f(x)>t\}.$ There are constants $C_1,C_2>0$ such that 
\[
 C_1\int_0^\infty \mathbf{Var} ( \mathbf{1}_{E_t(f)}) dt \le \mathbf{Var} (f) \le C_2 \int_0^\infty \mathbf{Var} ( \mathbf{1}_{E_t(f)}) dt .
\]
In particular, if $f \in BV(X)$ then for almost every $t$, $\mathbf{1}_{E_t(f)} \in BV(X)$. Conversely, if for almost every $t \ge 0$, $\mathbf{1}_{E_t(f)} \in BV(X)$ and $\int_0^\infty \mathbf{Var} ( \mathbf{1}_{E_t(f)}) dt < \infty$, then $f \in BV(X)$.
\end{theorem}

\begin{proof}
The first bound $ C_1\int_0^\infty \mathbf{Var} ( \mathbf{1}_{E_t(f)}) dt \le \mathbf{Var} (f) $ follows from Lemma~\ref{comparison var} and Lemma~\ref{lem:intVar*byVar*}.
For the second bound, we use Lemma~\ref{lem:Fubini for BV estimate} with $g\equiv1$ and the definition of the $\| \cdot \|_{1,1-\kappa/d_W}$ norm to obtain
\begin{equation*}
 \int\limits_X\int\limits_X p_s(x,y)|f(x)-f(y)|\,d\mu(x)d\mu(y)
 \leq 2 \int\limits_0^\infty \int\limits_X |(P_s-I)\mathbf{1}_{E_t(f)} |\,d\mu \,dt
 \leq 2 s^{1-\kappa/d_W} \int\limits_0^\infty \|\mathbf{1}_{E_t(f)}\|_{1,1-\kappa/d_W}\,dt.
 \end{equation*}
Dividing by $s^{1-\kappa/d_W}$ and taking $\liminf_{s\to0+}$ bound $\mathbf{Var}_*(f)$. The result then follows by dominating the integrand on the right using Lemma~\ref{Lemma 2:BV embed}.
\end{proof}

\subsection{Sets of finite perimeter}
Using the BV seminorm we can use the standard approach to define the perimeter of measurable sets.

\begin{definition}
Let $E \subset X$ be a Borel set. We say that $E$ has a finite perimeter if $\mathbf{1}_E \in BV(X)$. In that case, the perimeter of $E$ is defined as $P(E)=\mathbf{Var} (\mathbf{1}_E)$.
\end{definition}

The locality of the BV seminorm permits an improvement of the crude bound used in Lemma~\ref{lem:charfunctinBesov} to determine that a set had finite perimeter.  In essence, we can replace the upper Minkowski content used there with a corresponding $(d_W-\kappa)$-codimensional  \textit{lower} Minskowski content defined as follows:
\[
\mathcal{C}^*_{ d_W-\kappa} (E) =\liminf_{r \to 0^+} \frac{1}{r^{d_W-\kappa}} \mu (\partial^*_r E),
\]
with $\partial^*_r E$ as in~\eqref{eq:measure r neighborhood}.

\begin{theorem}\label{Minkowski} Under Assumption~\ref{assumpt:wBEatcrit} we have
\[
P(E) \le C \mathcal{C}^*_{d_W-\kappa} (E).
\]
In particular, any set for which $\cap_r \partial^*_r E$ has finite $(d_W-\kappa)$-codimensional lower Minskowski content has finite perimeter.
\end{theorem}
\begin{proof}
The following inequality was shown in the proof of Lemma~\ref{lem:charfunctinBesov}
\begin{equation*}
\iint_{\Delta_r}|\mathbf{1}_E(x)-\mathbf{1}_E(y)|\, d\mu(y)\, d\mu(x)
\leq Cr^{d_H}\mu(\partial^*_rE).
\end{equation*}
Dividing both sides of the inequality by $r^{d_H+d_W-\kappa}$ and taking $\liminf$ as $r\to 0^+$ yields the result.
\end{proof}




\subsection{Sobolev inequality}

It is well known that in the Euclidean space $\mathbb{R}^n$, there is a continuous embedding of BV into $L^{1^*}(\mathbb{R}^n)$, where $\frac{1}{1^*}=1-\frac{1}{n}$ is the critical $L^1$- Sobolev exponent. In our setting, using the results of our previous paper~\cite{ABCRST1}, one can prove a continuous embedding of $BV(X)$ into $L^{1^*}(X,\mu)$ where the critical Sobolev exponent $1^*$ is given by the formula
\[
\frac{1}{1^*}=1-\frac{d_W-\kappa}{d_H}.
\]
The relevant theorem is the following.

\begin{theorem}\label{Sobolev global}
Assume $d_W-\kappa < d_H$. Then $BV(X) \subset L^{1^*}(X,\mu)$ and there is $C>0$ such that for every $f \in BV(X)$,
\[
\| f \|_{L^{1^*}(X,\mu)} \le C \mathbf{Var}(f).
\]
In particular, there exists a constant $C>0$ such that for every set $E$ of finite perimeter,
\[
\mu(E)^{\frac{d_H-d_W+\kappa}{d_H}} \le C P(E).
\]
\end{theorem}

\begin{proof}
From the heat kernel upper bound \eqref{eq:subGauss-upper}, one has the ultracontractive estimate
$p_t(x,y) \le \frac{C}{t^{d_H/d_W}}$.
Moreover, from Theorem \ref{bounded embedding}, one has for every $f \in \B^{1,1-\frac{\kappa}{d_W}}(X)$,
\[
\|f\|_{1,1-\frac{\kappa}{d_W}} \le C\liminf_{s \to 0^+} \frac{1}{s^{1-\frac{\kappa}{d_W}}} \int_X P_s (|f-f(y)|)(y) d\mu(y).
\]
This verifies a condition denoted by $(P_{1,1-\frac{\kappa}{d_W}})$ in Definition~6.7 of~\cite{ABCRST1}, putting us in the framework of~\cite[Theorem 6.9]{ABCRST1} with $p=1$, $\alpha=1-\frac{\kappa}{d_W}$ and $\beta=\frac{d_H}{d_W}$. So we have
\[
\| f \|_{L^{1^*}(X,\mu)} \le C\|f\|_{1,1-\frac{\kappa}{d_W}}
\]
and the result follows from Theorem~\ref{bounded embedding}. 
\end{proof}

\begin{remark}
On the product space of nested fractals $(X^n, d_{X^n},\mu^{\otimes n})$, one has $\kappa=d_W-d_H(X)$. Thus the Sobolev exponent $1^*$ is given by
\[
\frac{1}{1^*}=1-\frac{d_W-(d_W-d_H(X))}{d_H(X^n)}=1-\frac1n,
\]
and the isoperimetric inequality becomes
\[
\mu(E)^{\frac{n-1}{n}} \le C P(E).
\]
These coincide with the Euclidean space $\mathbb R^n$.
\end{remark}

For the case $\kappa=d_W-d_H$ that corresponds to the situation in Theorem \ref{thm-xBE-FD}, 
 one has the following result. For $f \in L^\infty (X,\mu)$, we use $\mathbf{Osc} (f)$ to denote the essential supremum of $|f(x)-f(y)|$, $x,y \in X$.
\begin{proposition}\label{Sobolev 2}
Assume $\kappa=d_W-d_H>0$. Then $BV(X) \subset L^\infty(X,\mu)$ and there exists a constant $C>0$ such that for every $f \in BV(X)$,
\[
\mathbf{Osc} (f) \le C \mathbf{Var}(f).
\]
\end{proposition}

\begin{proof}
Let $f \in BV(X)$. Without loss of generality, we can assume $f \ge 0$ almost everywhere. For almost every $t \ge0 $ we define the set $E_t(f)=\{x\in X\, :\, f(x)>t\}$. Since $\kappa=d_W-d_H>0$, according to \cite[Corollary 6.6]{ABCRST1}, there is $c>0$ such that for every set $E$ of finite perimeter and positive measure, one has $P(E) \ge c$. 
However, from Theorem \ref{coarea formula 1a}, there is $C>0$, such that
\[
\int_0^\infty
 P(E_t(f))\, dt = \int_0^\infty \mathbf{Var} (\mathbf{1}_{E_t(f)}) \,dt \le C \mathbf{Var} (f ) <+\infty.
\]
Therefore the set $\Sigma (f)$ of $t$ values for which $\mu( E_t(f)) >0$ has finite Lebesgue measure. 
So from Fubini's theorem,
\[
\int_X \int_{\mathbb{R} \setminus \Sigma_f} \mathbf{1}_{E_t(f)}(x) dt d\mu(x) =\int_{\mathbb{R} \setminus \Sigma_f} \mu( E_t(f) ) dt =0
\]
and $\int_{\mathbb{R} \setminus \Sigma_f} \mathbf{1}_{E_t(f)}(x) dt=0$ almost everywhere.
Thus for almost every $x,y \in X$
\begin{align*}
| f(y)-f(x)| 
=\int_{0}^{+\infty} |\mathbf{1}_{E_t(f)}(x) -\mathbf{1}_{E_t(f)}(y) | dt 
&=\int_{\Sigma (f)} |\mathbf{1}_{E_t(f)}(x) -\mathbf{1}_{E_t(f)}(y) | dt \\
 & \le \frac{1}{c} \int_{\Sigma (f)} P(E_t(f)) dt 
 \le \frac{C}{c} \mathbf{Var} (f).\qedhere
\end{align*}
\end{proof}

\subsection{BV measures}\label{S:BV-measures}

Recall that for $f \in L^1(X,\mu)$ we set
\[
 \mathcal{M}_rf (y)=\frac{1}{r^{d_W-\kappa}\mu(B(y,r))} \int_{B(y,r)} |f(x) -f(y)| d\mu(x).
\]

\begin{definition}\label{def-BV-energy}
Let $f \in BV(X)$. A BV measure $\gamma_f$ is a Radon measure on $X$ such that 
there exists a sequence $r_n \searrow 0$, such that for every $g\in C_0(X)$,
\[
\lim_{n \to +\infty} \int_X g \mathcal{M}_{r_n} f d\mu = \int_X g d\gamma_f.
\]
\end{definition}

In other words, a BV measure is a cluster point of the family of measures $\mathcal{M}_{r} f d\mu$, $r>0$, in the vague 
topology on the space of Radon measures.

\begin{lemma}\label{lem:elementary}
 If $f \in BV(X)$, there exists at least one associated BV measure $\gamma_f$.
\end{lemma}
\begin{proof}
Let $\mu_r$ be the measure $d\mu_r := \mathcal{M}_r f \, d\mu$.
Note that for every $r>0$, $\mu_r (X) \le C \| f \|_{1, 1-\frac{\kappa}{d_W}}$, thus there is a sequence $r_n\to 0^+$ and a Radon measure $\gamma_f$ on $X$, such that $\mu_{r_n}$ converges vaguely to $\gamma_f$.
 
\end{proof}
%

\begin{remark}
\label{non uniqueness}
On certain fractal spaces it is known that the heat kernel has oscillations which preclude existence of a limiting density for a Weyl asymptotic~\cite[and references therein]{W15,WZ13,LPW11,Mattila,PW14,RZ12,Kajino,KigB,LP06}. Given the connection between $\mathbf{Var}(f)$ and $\mathbf{Var}_*(f)$, this suggests that measures of the type $\gamma_f$ may fail to be unique, but we do not study this phenomenon here. In the absence of uniqueness it is natural to consider upper and lower envelopes, which are discussed under certain extra assumptions in section~\ref{subsec-lower}, see in particular~\eqref{e-lower-env} and Theorem~\ref{thm-bv-kappa}. Upper envelope BV measures can be defined in a similar manner to~\eqref{e-lower-env}, but are omitted for the sake of brevity. \end{remark}


In the following, we denote by $\mathcal{D}$ the following class:
\[
\mathcal{D} =\left\{ P_\ve u: \ve>0 , u \in C_c(X) \right\}.
\]
Since the semigroup $P_t$ is Feller, $\mathcal{D}$ is dense for the supremum norm in $C_0(X)$. We note that from $wBE(\kappa)$, functions in $\mathcal{D}$ are $\kappa$-H\"older continuous.

\begin{theorem}\label{equivalence BV}
There exist constants $c,C>0$ such that for every $f \in BV(X) \cap L^\infty(X,\mu)$ and associated BV measure $\gamma_f$ we have for every $g \in \mathcal{D}$, $g \ge 0$,
\[
c \limsup_{t\to 0^+} \int_X g(y) \mathcal{Q}_tf (y) d\mu (y) \le \int_X g(y) d\gamma_f(y) \le C \liminf_{t\to 0^+} \int_X g(y) \mathcal{Q}_tf (y) d\mu (y),
\]
where 
\begin{align*}
\mathcal{Q}_t f (y) =\frac{1}{ t^{1-\frac{\kappa}{d_W}}} \int_Xp_{t}(x,y) |f(x)-f(y)| \,d\mu(x). 
\end{align*}

In particular, all the BV measures associated to a given $f$ are mutually equivalent with uniformly bounded densities.
\end{theorem}

The proof of the theorem is rather long and will be divided into several lemmas. 

\begin{lemma}\label{Lemma 2:BCD}
There exists a constant $C>0$ such that for every $f \in BV(X) \cap L^\infty (X,\mu)$ and every $g \in \mathcal{D}$,
\begin{align*}
 \limsup_{t\to 0^+} \frac{1}{t^{1-\frac{\kappa}{d_W}}} \left \| fg -P_t(fg) \right\|_{L^1(X,\mu)} \le C \liminf_{s\to 0^+} \int_X g(y) \mathcal{Q}_sf (y) d\mu (y) .
\end{align*}
\end{lemma}
\begin{proof}
As in the proof of Proposition~\ref{pseudo poincare}
\begin{equation}\label{eq:EtaufgPsh}
\biggl| \int_X \bigl( f g-P_t(f g) \bigr)h\,d\mu\biggr|
\leq \liminf_{\tau\to0+} \int_0^t \bigl|\mathcal{E}_\tau(f g, P_sh)\bigr| \,ds.
\end{equation}

Recall $g=P_\ve u$ for some $u\in C_c(X)$ and $\ve>0$. Then, from a standard energy calculation followed by using the $wBE(\kappa)$ estimate and~\eqref{eq:boundpolyptbypct} as in the proof of Proposition~\ref{pseudo poincare}, we may write
\begin{align*}
 \lefteqn{ \bigl| \DF_\tau( gP_sh,f)+\DF_\tau(fP_sh, g)-\DF_\tau(f g,P_sh)\bigr|}\quad&\\
 &=\frac1{\tau}\biggl|\int_X\int_X p_\tau(x,y) (f(x)-f(y))( g(x)- g(y))P_sh(x) d\mu(x)d\mu(y)\biggr|\\
 &\leq C\|h\|_\infty \|u\|_\infty \Bigl(\frac\tau\ve\Bigr)^{\kappa/d_W} \frac{1}{\tau}\int_X\int_X p_{c\tau}(x,y)|f(x)-f(y)|d\mu(x)d\mu(y)\\
 &\leq C\|h\|_\infty \|u\|_\infty \ve^{-\kappa/d_W} \|f\|_{1,1-\kappa/d_W}.
\end{align*}
Almost the same argument, now also using $g\geq0$, shows
\begin{align*}
 \lefteqn{ \bigl| \DF_\tau( gP_sh,f)+\DF_\tau(f g,P_sh)-\DF_\tau(fP_sh, g)\bigr|}\quad&\\
 &=\frac1{\tau} \biggl|\int_X\int_X p_\tau(x,y) (f(x)-f(y))(P_sh(x)-P_sh(y)) g(x) d\mu(x)d\mu(y)\biggr|\\
 &\leq C\|h\|_\infty \Bigl(\frac\tau s\Bigr)^{\kappa/d_W}\frac1{\tau} \int_X\int_X p_{c\tau}(x,y)|f(x)-f(y)| g(x) d\mu(x)d\mu(y)\\
 &= C\|h\|_\infty s^{-\kappa/d_W} \int_X g(y) \mathcal{Q}_{c\tau} f(y)\,d\mu(y).
\end{align*}
Taking the difference of these expressions, integrating with respect to $s$ and taking $\liminf_{\tau\to0^+}$ we find 
\begin{align*}
 \lefteqn{\liminf_{\tau\to0^+} 2\int_0^t \bigl|\DF_\tau(f g,P_sh)- \DF_\tau (fP_sh, g)\bigr|\,ds}\quad&\\
 &\leq Ct \|h\|_\infty \|u\|_\infty \ve^{-\kappa/d_W} \|f\|_{1,1-\kappa/d_W} + C\|h\|_\infty t^{1-\kappa/d_W} \liminf_{\tau\to0^+} \int_X g(y) \mathcal{Q}_\tau f (y)\,d\mu(y).
 \end{align*}
However $ g$ is in the $L^1$ domain of $L$ so $\lim_{\tau\to0^+}\DF_\tau (fP_sh, g)=\int_X (L g)fP_sh$ and thus $|\DF_\tau (fP_sh, g)|\leq 2 \|fP_sh\|_{L^\infty(X,\mu)}\|L g\|_{L^1(X,\mu)}$ for all sufficiently small $\tau$, independent of $s$. In particular
\begin{align*}
 \liminf_{\tau\to0^+} 2\int_0^t \bigl|\DF_\tau(f g,P_sh)\bigr|\,ds
 &\leq Ct \|h\|_{L^\infty(X,\mu)} \Bigl( \|f\|_{L^\infty(X,\mu)} \|L g\|_{L^1(X,\mu)} + \ve^{-\kappa/d_W} \| u \|_\infty \|f\|_{1,1-\kappa/d_W}\Bigr)\\
 &\quad+ C\|h\|_\infty t^{1-\kappa/d_W}\liminf_{\tau\to0^+}\int_X g(y)\mathcal{Q}_\tau f(y)\,d\mu(y).
 \end{align*}
Comparing this to~\eqref{eq:EtaufgPsh}, dividing by $t^{1-\kappa/d_W}$ and taking $\limsup_{t\to0^+}$ complete the proof by $L^1$-$L^\infty$ duality.
\end{proof}

\begin{remark}\label{rem:BCD2general}
The conclusion of Lemma \ref{Lemma 2:BCD} also holds for $f=f_0+c$, where $f_0\in BV(X) \cap L^{\infty}(X,\mu)$ and $c$ is a constant. Indeed, note that $\|f\|_{1, 1-\kappa/d_W}=\|f_0\|_{1, 1-\kappa/d_W}$ and $\mathcal{Q}_\tau f(y)=\mathcal{Q}_\tau f_0(y)$, the above argument then applies.
\end{remark}

\begin{lemma}\label{Lemma 3:BCD}
There is $C>0$ such that for $f \in BV(X)\cap L^\infty(X,\mu)$, $f \ge 0$, we have for every $g \in \mathcal{D}$, $g \ge 0$,
\[
\limsup_{t\to 0^+} \int_X g(y) \mathcal{Q}_tf (y) d\mu (y) \le C \liminf_{t\to 0^+} \int_X g(y) \mathcal{Q}_tf (y) d\mu (y) .
\]
\end{lemma}

\begin{proof}
With $M=\|f\|_\infty$, use Lemma~\ref{lem:Fubini for BV estimate} and  Lemma~\ref{Lemma 2:BCD}, Remark \ref{rem:BCD2general}, the reverse Fatou lemma, and the Fatou lemma to deduce
 \begin{align*}
 & \limsup_{s \to 0^+} \frac{1}{s^{1-\frac{\kappa}{d_W}}} \int_X \int_X g(y) p_s(x,y) |f(x)-f(y)| d\mu(x) d\mu(y) 
 \\
 \le &  \int_{0}^M \limsup_{s\to 0^+}\frac{1}{s^{1-\frac{\kappa}{d_W}}} \biggl( \| P_s(\mathbf{1}_{E_t(f)}g) - \mathbf{1}_{E_t(f)}g \|_{L^1(X,\mu)}+\| P_s(\mathbf{1}_{X\setminus E_t(f)}g) - \mathbf{1}_{X\setminus E_t(f)}g \|_{L^1(X,\mu)} \biggr)dt 
 \\
 \le & C \int_{0}^M \liminf_{s\to 0^+}
\frac{1}{ s ^{1-\frac{\kappa}{d_W}}}\int_X \int_X g(y) p_{s}(x,y) |\mathbf{1}_{E_t(f)}(x)-\mathbf{1}_{E_t(f)}(y)| \,d\mu(x) \,d\mu(y) \, dt \\
\le & C \liminf_{s\to 0^+}
\frac{1}{ s ^{1-\frac{\kappa}{d_W}}} \int_{0}^M \int_X \int_X g(y) p_{s}(x,y) |\mathbf{1}_{E_t(f)}(x)-\mathbf{1}_{E_t(f)}(y)| \,d\mu(x) \,d\mu(y) \, dt \\
\le & C \liminf_{s\to 0^+}
\frac{1}{ s ^{1-\frac{\kappa}{d_W}}} \int_X \int_X g(y) p_{s}(x,y) |f(x)-f(y)| \,d\mu(x) \,d\mu(y) \, dt.\qedhere
\end{align*}
\end{proof}

\begin{lemma}\label{estimate BV measure 1}
There is $C>0$ such that for $f \in BV(X)$ and $g \in C_0(X),g\ge 0$, one has for $t>0$
\[
\int_X g(y) \mathcal{Q}_t f (y) d\mu (y) \ge C \int_X g(y) \mathcal{M}_{t^{1/d_W}}f (y) d\mu (y).
\]
\end{lemma}

\begin{proof}
For any $t>0$, by the sub-Gaussian heat kernel lower bound we have $p_t(x,y)\geq C t^{-d_H/d_W}$ on $B(y,t^{1/d_W})$, so
\begin{align*}
 \frac1{t^{1-\frac{\kappa}{d_W}}} &\int_X \int_X |f(x)-f(y) | p_t (x,y) d\mu(x) g(y) d\mu(y) \\
 \ge & \frac1{t^{1-\frac{\kappa}{d_W}}} \int_X \int_{B(y,t^{1/d_W})} |f(x)-f(y) | p_t (x,y) d\mu(x) g(y) d\mu(y) \\
 \ge & \frac C{t^{\frac{d_{H}}{d_{W}}+1-\frac{\kappa}{d_W}}} \int_X \int_{B(y,t^{1/d_W})} |f(x)-f(y) | d\mu(x) g(y)d\mu(y)
 = C \int_X g(y) \mathcal{M}_{t^{1/d_W}}f (y) d\mu (y).\qedhere
\end{align*}
\end{proof}

\begin{lemma}\label{estimate BV measure 4}
There exists a constant $C>0$ such that for every $f \in BV(X)\cap L^\infty(X,\mu)$ and $ g \in \mathcal{D},g\ge 0$
\[
\liminf_{t\to 0^+} \int_X g(y) \mathcal{Q}_tf (y) d\mu (y) \le C \liminf_{r\to 0^+} \int_X g(y) \mathcal{M}_r f (y) d\mu (y) .
\]
\end{lemma}

\begin{proof}
The proof is similar to that of Lemma~\ref{comparison var}. Let
\begin{equation*}
	\Psi_g(t) = \frac{1}{ t^{1-\frac{\kappa}{d_W}}}\int_X\int_X g(y) p_{t}(x,y) |f(x)-f(y)| \,d\mu(x)\,d\mu(y) 
	\end{equation*}
and estimate by dividing the region of integration and using the sub-Gaussian bounds to obtain, as was done in~\eqref{eq:ineqPhiandPsi},
\begin{equation*}
	\Psi_g(t) \leq C \delta^{d_W+d_H-\kappa} \int_X g(y) \mathcal{M}_{\delta t^{1/d_W}} f (y) d\mu(y)+ \exp\Bigl(-c \delta^{\frac{d_W}{d_{W}-1}}\Bigr) \Psi_g(ct).
	\end{equation*}
The proof can then be completed in the same manner as Lemma~\ref{comparison var}.
\end{proof}

We are finally in position to prove Theorem \ref{equivalence BV}.

\begin{proof}[Proof of Theorem \ref{equivalence BV}]
Without loss of generality, we can assume $f \ge 0$ a.e. From Lemma \ref{Lemma 3:BCD},
\[
\limsup_{t\to 0^+} \int_X g(y) \mathcal{Q}_tf (y) d\mu (y) \le C \liminf_{t\to 0^+} \int_X g(y) \mathcal{Q}_tf (y) d\mu (y).
\]
Then, from Lemma \ref{estimate BV measure 4},
\[
\liminf_{t\to 0^+} \int_X g(y) \mathcal{Q}_tf (y) d\mu (y) \le C \liminf_{r\to 0^+} \int_X g(y) \mathcal{M}_r f (y) d\mu (y) .
\]
Finally, Lemma \ref{estimate BV measure 1} implies
\[
\limsup_{r\to 0^+} \int_X g(y) \mathcal{M}_r f (y) d\mu (y) \le C\limsup_{t\to 0^+} \int_X g(y) \mathcal{Q}_tf (y) d\mu (y).
\]
\end{proof}

\subsection{Lower estimates and the relation between BV and energy measures}\label{subsec-lower}

In this section, we compare the BV measures with energy measures for functions which are both in $BV(X)$ and the domain $\mathcal{F}$ of the form. For this, we recall the notion of $\liminf$ measure used by M. Sion in~\cite{Sion64}.

\begin{definition}[\cite{Sion64}]
Let $(\mu_r)_{r >0}$ be a family of Radon measures on $X$. The $\liminf$ measure $\underline{\mu}$ of the family $(\mu_r)_{r >0}$ is defined as
\begin{equation}\label{e-lower-env}
\underline{\mu} (A)= \inf_{U \text{ open}, A \subset U} \sup_{K \text{ compact}, K \subset U } \underline{\mu}^* (K),
\end{equation} 
where 
\begin{equation*}
\underline{\mu}^* (K)=\inf \left\{ \sum_i \liminf_{r \to 0^+} \mu_r (U_i): U_i \text{ open}, K \subset \bigcup_i U_i \right\}.
\end{equation*}
It satisfies the following property: If $(\nu_r)_{r>0}$ is a family of Radon measures such that 
\begin{itemize}
\item $\nu_r \le \mu_r$;
\item $\nu_r$ vaguely converges to some Radon measure $\nu$.
\end{itemize}
Then $\nu \le \underline{\mu}$.
\end{definition}

We now recall that since $\mathcal{E}$ is assumed to be regular, for every $f\in \mathcal F\cap L^{\infty}(X,\mu)$, one can define the energy measure $\nu_{f}$ in the sense of~\cite{Beurling-Deny} through the formula
\[
\int_X \phi d\nu_{f} = \mathcal{E}(f\phi,f)-\frac12 \mathcal{E}(\phi,f^2), \quad \phi\in \mathcal F \cap C_c(X).
\]
Then $\nu_{f}$ can be extended to all $f\in \mathcal F$ by truncation. 
\begin{theorem}\label{thm-bv-kappa}
If $f \in BV(X) \cap L^2(X,\mu)$ is H\"older continuous with exponent $\kappa$, then $f \in \mathcal{F}$ and its energy measure $\nu_f$ is absolutely continuous with respect to any BV measure $\gamma_f$. Moreover, there exists a constant $C>0$ independent from $f$ such that
\[
\nu_f \le C \| f \|_{\infty,\kappa} \underline{\gamma}_f ,
\]
where $\| f \|_{\infty,\kappa}$ denotes the $\kappa$-H\"older seminorm of $f$ and $\underline{\gamma}_f$ the $\liminf$ measure of the family $\mathcal{M}_r f d\mu$, $r >0$. 
\end{theorem}

\begin{proof}
Suppose that $f \in BV(X)$ is H\"older continuous with exponent $\kappa$.
Then by~\eqref{eq:boundpolyptbypct}, for $g \in C_0(X)$,
\begin{align*}
 \lefteqn{\frac1t\int_{X}\int_X g(y) | f(x) -f(y)|^2 p_t (x,y) \,d\mu(x)d\mu(y)}\quad& \\
 \le & \frac Ct \| f \|_{\infty,\kappa} \int_{X}\int_X g(y) d(x,y)^\kappa | f(x) -f(y)| p_t (x,y) \,d\mu(x)d\mu(y) \\
 \le &  \| f \|_{\infty,\kappa} \frac C{t^{1-\kappa/d_{W}}} \int_X \int_X g(y) p_{Ct} (x,y) |f(x)-f(y)| \,d\mu(x) d\mu(y).
\end{align*}
The fact that $\frac{1}{2t} | f(x) -f(y)|^2 p_t (x,y) \,d\mu(x)$ converges vaguely to $\nu_f$ as $t\to0^+$ shows $\nu_f$ can be dominated by the lim inf measure of $\mathcal{Q}_t f d\mu$, and thus, applying~Theorem \ref{equivalence BV}, by  $\underline{\gamma}_f$.  In particular, we recover $\DF(f,f)\leq C \| f \|_{\infty,\kappa} \mathbf{Var}(f)$.
\end{proof}

\section{Examples}\label{S:Examples}
We conclude the paper with some results and conjectures about the BV class for two explicit classes  of examples, the nested fractals 
\cite{Lind,FitzsimmonsHamblyKumagai} and generalized Sierpinski carpets \cite{BB92,BaASC,BBKT,BB99,BB89}.  
\subsection{Fractional spaces}

Nested fractals represent an important class of examples of finitely ramified, self-similar, fractional metric spaces that support fractional diffusions in the sense of Barlow (\cite[Definition 3.2]{Ba98},  and see also Section~\ref{ssec:Barlowspaces}).  We omit the technical definition for the sake of brevity, but note that two standard  nested fractals that   exemplify the behavior seen in this class are the Vicsek set and the Sierpinski gasket shown in 
Figures~\ref{fig-Vicsek} and~\ref{fig-SG}. 

Our BV theory applies to these examples in its entirety, i.e.\ Assumption~\ref{assumpt:wBEatcrit} is satisfied.  In particular, we have a weak Bakry-\'Emery inequality $wBE(\kappa)$ with $\kappa=d_W-d_H>0$ by virtue of Theorem~\ref{thm-xBE-FD}.  According to Theorem~\ref{estimate kappa alpha}  we then have $\alpha_1^*\leq \beta_1=1-\kappa/d_W=d_H/d_W$.  However the converse is established by the following theorem.

\begin{theorem}\label{T:nested}
If $X$ is a nested fractal, then $\B^{1,d_H/d_W}(X)$ contains all indicator functions of cells with finite boundary, so is dense in  $L^1(X,\mu)$. Hence $\alpha_1^* = d_H/d_W$.
\end{theorem}
\begin{proof}
This result follows easily from Lemma~\ref{lem:charfunctinBesov} because $\partial^*_rE$ consists of $r$-neighborhoods of the finite number of boundary points so $\mu(\partial^*_rE)\leq Cr^{d_H}$ for some $C$ depending on the set. With $\alpha=d_H/d_W$ it is then clear that $r^{-\alpha/d_W}\mu(\partial^*_rE)$ is bounded.
\end{proof}
 
 We have thus established that $\kappa=d_W(1-\alpha_1^*)$, so we are in the setting of Section~\ref{S:BVfunctions}, with $BV(X)=\B^{1,d_H/d_W}(X)$ and
 \begin{equation*}
 	\mathbf{Var}(f) = \liminf_{r \to 0^+} \frac1{r^{d_H}} \int_X\frac1{\mu(B(x,r))} \int_{B(x,r)} |f(x)-f(y)|\,d\mu(y).
 	\end{equation*} 
 BV functions in these examples therefore have all of basic  properties seen in this paper: locality of BV norms, co-area estimate, control of perimeter measures by lower Minkowski content, Sobolev inequalities, and BV measures.  A detailed presentation of the theory in this setting will be given in~\cite{ABCRST4}, to which we refer the reader for the proof of the following result about piecewise harmonic functions that gives an indication of how different the BV theory on fractals may be to that on Euclidean spaces.  Piecewise harmonic functions are functions that are continuous and harmonic except at a finite set of points.
  They are the analogue of piecewise linear functions on the real line.
 
 \begin{theorem}[See \cite{ABCRST4}] \label{T:Vicsek-SG}
 \
 
\begin{enumerate}
\item \label{T:Vicsek-SG1}
On the Vicsek set any  compactly supported piecewise harmonic  function is in BV, and its BV measures are equivalent to its energy measure. 
\item \label{T:Vicsek-SG2}
On the Sierpinski gasket any non-constant piecewise harmonic function is not in BV. 
\end{enumerate}\end{theorem}
 
We believe that stronger results are true, namely that
\begin{conjecture}\label{conj-SG}\ 
\begin{enumerate}
\item On a nested fractal that is a dendrite, such as the Vicsek set, the BV space can be completely described using an analogue of Stieltjes integration along geodesics, and hence BV functions may be described as having classical distributional derivatives that are finite Radon measures.
\item On the Sierpinski gasket, and certain other nested fractals, any non-constant continuous function is of infinite variation.
\end{enumerate}
\end{conjecture}

We make a brief comment about fractional metric spaces that support a fractional diffusion but are not nested fractals.  The most basic examples of such spaces are generalized Sierpinski carpets, for which 
there exist a unique Dirichlet form and diffusion process, see \cite{BBKT}, and 
there is a comprehensive description of properties of the heat operator due to Barlow and Bass~\cite{BB89,BB99}.   It is not difficult to use Lemma~\ref{lem:charfunctinBesov} to see that the critical exponent \begin{equation}\label{e-SC-a*}
\alpha_1^*\geq (d_H-d_{tH}+1)/d_W,
\end{equation} where $d_{tH}$ is the {topological-Hausdorff dimension} defined in~\cite{BBE15}. For the classical Sierpinski carpet $d_{tH}=\dfrac{\log 2}{\log 3}+1$ according to~\cite[Theorem 5.4]{BBE15}. However the Barlow-Bass theory only yields $wBE(\kappa)$ for $\kappa=d_W-d_H$, not for $\kappa= d_W-d_H + d_{tH}-1$.  We believe equality holds in \eqref{e-SC-a*} for $\alpha_1^*$ and post an open question about the   weak Bakry-\'Emery estimate at criticality. Note that,   if $1<d_S=2\frac{d_H}{d_W}<2$,   proving $wBE(\kappa)$ for $ \kappa > d_W-d_H $ would involve improving the H\"older continuity estimates for harmonic functions  in~\cite{BB89,BB99,Ba98}. 
Improved H\"older continuity estimates for harmonic functions on the classical two dimensional Sierpinski carpet are strongly supported by numerical calculations in \cite{REU}.
\begin{conjecture}\label{conj-SC}
We conjecture that for generalized Sierpinski carpets and similar fractals $$\alpha_1^*=(d_H-d_{tH}+1)/d_W$$ and the condition $wBE(\kappa)$ is valid 
for some $ \kappa > (d_W-d_H)_+   $. 
\end{conjecture}
Note that this conjecture together with Theorem~\ref{estimate kappa alpha} implies that $    \kappa \leqslant  d_W-d_H + d_{tH}-1  $, which makes the following question natural and important. 
\begin{open question}\label{question-SC}
Investigate under which conditions $wBE(d_W-d_H + d_{tH}-1)$    holds true and Assumption~\ref{assumpt:wBEatcrit} is satisfied. 
\end{open question}

\subsection{BV functions in products of nested fractals}
Further  examples 
of spaces to which our  theory applies 
can be constructed by taking products of nested fractals.  
In some sense these are the most interesting class of examples because one should expect that there are subsets with a non-trivial notion of curvature, however here we only discuss the product spaces.
Suppose $X$ is a nested fractal of dimension $d_H$ on which the diffusion has walk dimension $d_W$.  The condition $wBE(d_W-d_H)$ is valid by Theorem~\ref{thm-xBE-FD}.  The $n$-fold product $X^n$ supports a heat kernel obtained by tensoring and discussed in  Section~\ref{subsec-products}, where it was established  that the walk dimension remains $d_W$ on the product and $wBE(d_W-d_H)$ is still true.  All that has changed is that the Hausdorff dimension is now $nd_H$.
\begin{theorem}\label{T:nested product}
If $X$ is a nested fractal, then for every $n\in\mathbb{N}$, the space $BV(X^n)=\B^{1,d_H/d_W}(X^n)$ is dense in $L^1(X^n,\mu^{\otimes n})$ and Assumption~\ref{assumpt:wBEatcrit} is satisfied.
\end{theorem}
\begin{proof}
Observe that the collection of sets $E=\prod_1^n E_j$, where each $E_j\subset X$ is a cell, generates the topology.  Moreover the boundary of such a set $E$ is a finite collection of faces of the form $$\Bigl(\prod_1^{k-1} E_j\Bigr)\times\{x_k\}\times \Bigl(\prod_{k+1}^n E_j\Bigr).$$  Each such face has (Hausdorff and, by self-similarity of the nested fractal, Minkowski) dimension $(n-1)d_H$, so there is $C$ such that for each $r>0$ it can be covered by $Cr^{-(n-1)d_H}$ balls of radius $r$; by doubling the radius of each ball we may ensure we cover an $r$-neighborhood of the face.  Each such ball has measure at most $c_2(2r)^{n d_H}$ by Ahlfors regularity, so the total measure involved in covering an $r$-neighborhood of the face is $Cc_2 2^{d_H} r^{d_H}$.  Summing over the finite number of faces we find that $\mu((\partial E)_r)\leq Cr^{d_H}$, so the result follows from Corollary~\ref{cor:conditforbesovdensity}.
It follows that $\alpha_1^*\geq \frac{d_H}{d_W}$, and since we know $wBE(d_W-d_H)$ implies $\alpha_1^*\leq\frac{d_H}{d_W}$ we have $\alpha_1^*=\frac{d_H}{d_W}$ and the weak Bakry-\'Emery condition is valid at the critical exponent, see Example~\ref{example-nested} and Remark~\ref{rem-wBE*}.   
\end{proof}

 

\bibliographystyle{plain}
 \bibliography{BV_Refs}

\begin{thebibliography}{10}

\bibitem{ABCRST6}
P.~Alonso-Ruiz, F.~Baudoin, L.~Chen, L.~Rogers, N.~Shanmugalingam, and
  A.~Teplyaev.
\newblock {BV} functions and fractional {L}aplacians on {D}irichlet spaces.
\newblock {\em preprint arXiv:1910.13330}, 2020.

\bibitem{ABCRST4}
P.~Alonso-Ruiz, F.~Baudoin, L.~Chen, L.~Rogers, N.~Shanmugalingam, and
  A.~Teplyaev.
\newblock Besov class via heat semigroup on {D}irichlet spaces {IV}: {N}ested
  fractals.
\newblock {\em preprint}, 2021.

\bibitem{ABCRST5}
P.~Alonso-Ruiz, F.~Baudoin, L.~Chen, L.~Rogers, N.~Shanmugalingam, and
  A.~Teplyaev.
\newblock Besov class via heat semigroup on {D}irichlet spaces {V}: {BV}
  functions in infinite-dimensional spaces.
\newblock {\em preprint}, 2021.

\bibitem{ABCRST1}
P.~Alonso~Ruiz, F.~Baudoin, L.~Chen, L.~G. Rogers, N.~Shanmugalingam, and
  A.~Teplyaev.
\newblock Besov class via heat semigroup on {D}irichlet spaces {I}: {S}obolev
  type inequalities.
\newblock {\em J. Funct. Anal.}, 278(11):108459, 48, 2020.

\bibitem{ABCRST2}
P.~Alonso~Ruiz, F.~Baudoin, L.~Chen, L.~G. Rogers, N.~Shanmugalingam, and
  A.~Teplyaev.
\newblock Besov class via heat semigroup on {D}irichlet spaces {II}: {BV}
  functions and {G}aussian heat kernel estimates.
\newblock {\em Calc. Var. Partial Differential Equations}, 59(3):Paper No.103,
  32, 2020.

\bibitem{AMP}
L.~Ambrosio, M.~Miranda, Jr., and D.~Pallara.
\newblock Special functions of bounded variation in doubling metric measure
  spaces.
\newblock In {\em Calculus of variations: topics from the mathematical heritage
  of {E}.\ {D}e {G}iorgi}, volume~14 of {\em Quad. Mat.}, pages 1--45. Dept.
  Math., Seconda Univ. Napoli, Caserta, 2004.

\bibitem{BGS14}
Luigi Ambrosio, Nicola Gigli, and Giuseppe Savar\'{e}.
\newblock Calculus and heat flow in metric measure spaces and applications to
  spaces with {R}icci bounds from below.
\newblock {\em Invent. Math.}, 195(2):289--391, 2014.

\bibitem{BGL}
Dominique Bakry, Ivan Gentil, and Michel Ledoux.
\newblock {\em Analysis and geometry of {M}arkov diffusion operators}, volume
  348 of {\em Grundlehren der Mathematischen Wissenschaften [Fundamental
  Principles of Mathematical Sciences]}.
\newblock Springer, Cham, 2014.

\bibitem{BBE15}
Rich\'{a}rd Balka, Zolt\'{a}n Buczolich, and M\'{a}rton Elekes.
\newblock A new fractal dimension: the topological {H}ausdorff dimension.
\newblock {\em Adv. Math.}, 274:881--927, 2015.

\bibitem{BaASC}
M.~T. Barlow.
\newblock Analysis on the {S}ierpinski carpet.
\newblock In {\em Analysis and geometry of metric measure spaces}, volume~56 of
  {\em CRM Proc. Lecture Notes}, pages 27--53. Amer. Math. Soc., Providence,
  RI, 2013.

\bibitem{BCG01}
Martin Barlow, Thierry Coulhon, and Alexander Grigor'yan.
\newblock Manifolds and graphs with slow heat kernel decay.
\newblock {\em Invent. Math.}, 144(3):609--649, 2001.

\bibitem{Ba98}
Martin~T. Barlow.
\newblock Diffusions on fractals.
\newblock In {\em Lectures on probability theory and statistics
  ({S}aint-{F}lour, 1995)}, volume 1690 of {\em Lecture Notes in Math.}, pages
  1--121. Springer, Berlin, 1998.

\bibitem{Ba03}
Martin~T. Barlow.
\newblock Heat kernels and sets with fractal structure.
\newblock In {\em Heat kernels and analysis on manifolds, graphs, and metric
  spaces ({P}aris, 2002)}, volume 338 of {\em Contemp. Math.}, pages 11--40.
  Amer. Math. Soc., Providence, RI, 2003.

\bibitem{Barlow04}
Martin~T. Barlow.
\newblock Which values of the volume growth and escape time exponent are
  possible for a graph?
\newblock {\em Rev. Mat. Iberoamericana}, 20(1):1--31, 2004.

\bibitem{BB89}
Martin~T. Barlow and Richard~F. Bass.
\newblock The construction of {B}rownian motion on the {S}ierpi\'nski carpet.
\newblock {\em Ann. Inst. H. Poincar\'e Probab. Statist.}, 25(3):225--257,
  1989.

\bibitem{BB92}
Martin~T. Barlow and Richard~F. Bass.
\newblock Transition densities for {B}rownian motion on the {S}ierpi{\'n}ski
  carpet.
\newblock {\em Probab. Theory Related Fields}, 91(3-4):307--330, 1992.

\bibitem{BB99}
Martin~T. Barlow and Richard~F. Bass.
\newblock Brownian motion and harmonic analysis on {S}ierpinski carpets.
\newblock {\em Canad. J. Math.}, 51(4):673--744, 1999.

\bibitem{BB04}
Martin~T. Barlow and Richard~F. Bass.
\newblock Stability of parabolic {H}arnack inequalities.
\newblock {\em Trans. Amer. Math. Soc.}, 356(4):1501--1533, 2004.

\bibitem{BBK06}
Martin~T. Barlow, Richard~F. Bass, and Takashi Kumagai.
\newblock Stability of parabolic {H}arnack inequalities on metric measure
  spaces.
\newblock {\em J. Math. Soc. Japan}, 58(2):485--519, 2006.

\bibitem{BBKT}
Martin~T. Barlow, Richard~F. Bass, Takashi Kumagai, and Alexander Teplyaev.
\newblock Uniqueness of {B}rownian motion on {S}ierpi\'nski carpets.
\newblock {\em J. Eur. Math. Soc. (JEMS)}, 12(3):655--701, 2010.

\bibitem{BP}
Martin~T. Barlow and Edwin~A. Perkins.
\newblock Brownian motion on the {S}ierpi\'nski gasket.
\newblock {\em Probab. Theory Related Fields}, 79(4):543--623, 1988.

\bibitem{BGN01}
Laurent Bartholdi, Rostislav Grigorchuk, and Volodymyr Nekrashevych.
\newblock From fractal groups to fractal sets.
\newblock In {\em Fractals in {G}raz 2001}, Trends Math., pages 25--118.
  Birkh\"{a}user, Basel, 2003.

\bibitem{B-JEMS}
Richard~F. Bass.
\newblock A stability theorem for elliptic {H}arnack inequalities.
\newblock {\em J. Eur. Math. Soc. (JEMS)}, 15(3):857--876, 2013.

\bibitem{BB2}
Fabrice Baudoin and Michel Bonnefont.
\newblock Reverse {P}oincar\'e inequalities, isoperimetry, and {R}iesz
  transforms in {C}arnot groups.
\newblock {\em Nonlinear Anal.}, 131:48--59, 2016.

\bibitem{Ba-Ke19}
Fabrice Baudoin and Daniel~J. Kelleher.
\newblock Differential one-forms on {D}irichlet spaces and {B}akry-\'{E}mery
  estimates on metric graphs.
\newblock {\em Trans. Amer. Math. Soc.}, 371(5):3145--3178, 2019.

\bibitem{BK14}
Fabrice Baudoin and Bumsik Kim.
\newblock Sobolev, {P}oincar\'e, and isoperimetric inequalities for subelliptic
  diffusion operators satisfying a generalized curvature dimension inequality.
\newblock {\em Rev. Mat. Iberoam.}, 30(1):109--131, 2014.

\bibitem{BenBassatStrichartzTeplyaev}
Oren Ben-Bassat, Robert~S. Strichartz, and Alexander Teplyaev.
\newblock What is not in the domain of the {L}aplacian on {S}ierpinski gasket
  type fractals.
\newblock {\em J. Funct. Anal.}, 166(2):197--217, 1999.

\bibitem{Beurling-Deny}
A.~Beurling and J.~Deny.
\newblock Espaces de {D}irichlet. {I}. {L}e cas \'{e}l\'{e}mentaire.
\newblock {\em Acta Math.}, 99:203--224, 1958.

\bibitem{St07p}
Brian Bockelman and Robert~S. Strichartz.
\newblock Partial differential equations on products of {S}ierpinski gaskets.
\newblock {\em Indiana Univ. Math. J.}, 56(3):1361--1375, 2007.

\bibitem{REU}
Claire Canner, Christopher Hayes, William Huang, Michael Orwin, and Luke~G.
  Rogers.
\newblock Improved {H}\"older continuity estimates for harmonic functions on
  the {S}ierpinski carpet.
\newblock {\em preprint}.

\bibitem{CaMo18}
Fabio Cavalletti and Andrea Mondino.
\newblock Isoperimetric inequalities for finite perimeter sets under lower
  {R}icci curvature bounds.
\newblock {\em Atti Accad. Naz. Lincei Rend. Lincei Mat. Appl.},
  29(3):413--430, 2018.

\bibitem{Chee99}
Jeff Cheeger.
\newblock Differentiability of {L}ipschitz functions on metric measure spaces.
\newblock {\em Geom. Funct. Anal.}, 9:428--517, 1999.

\bibitem{ChenFukushima}
Zhen-Qing Chen and Masatoshi Fukushima.
\newblock {\em Symmetric {M}arkov processes, time change, and boundary theory},
  volume~35 of {\em London Mathematical Society Monographs Series}.
\newblock Princeton University Press, Princeton, NJ, 2012.

\bibitem{Coulhon03}
Thierry Coulhon.
\newblock Off-diagonal heat kernel lower bounds without {P}oincar\'{e}.
\newblock {\em J. London Math. Soc. (2)}, 68(3):795--816, 2003.

\bibitem{Da97}
E.~B. Davies.
\newblock Non-{G}aussian aspects of heat kernel behaviour.
\newblock {\em J. London Math. Soc. (2)}, 55(1):105--125, 1997.

\bibitem{DeGiorgi54}
Ennio De~Giorgi.
\newblock Su una teoria generale della misura {$(r-1)$}-dimensionale in uno
  spazio ad {$r$} dimensioni.
\newblock {\em Ann. Mat. Pura Appl. (4)}, 36:191--213, 1954.

\bibitem{FitzsimmonsHamblyKumagai}
Pat~J. Fitzsimmons, Ben~M. Hambly, and Takashi Kumagai.
\newblock Transition density estimates for {B}rownian motion on affine nested
  fractals.
\newblock {\em Comm. Math. Phys.}, 165(3):595--620, 1994.

\bibitem{FOT}
Masatoshi Fukushima, Yoichi Oshima, and Masayoshi Takeda.
\newblock {\em Dirichlet forms and symmetric {M}arkov processes}, volume~19 of
  {\em De Gruyter Studies in Mathematics}.
\newblock Walter de Gruyter \& Co., Berlin, extended edition, 2011.

\bibitem{gong2018li}
Chao Gong, Yong Lin, Shuang Liu, and Shing-Tung Yau.
\newblock Li-{Y}au inequality for unbounded {L}aplacian on graphs.
\newblock {\em Adv. Math.}, 357:106822, 23, 2019.

\bibitem{GK08}
A.~Grigor'yan and T.~Kumagai.
\newblock On the dichotomy in the heat kernel two sided estimates.
\newblock In {\em Analysis on graphs and its applications}, volume~77 of {\em
  Proc. Sympos. Pure Math.}, pages 199--210. Amer. Math. Soc., Providence, RI,
  2008.

\bibitem{Gri}
Alexander Grigor'yan.
\newblock Heat kernels and function theory on metric measure spaces.
\newblock In {\em Heat kernels and analysis on manifolds, graphs, and metric
  spaces ({P}aris, 2002)}, volume 338 of {\em Contemp. Math.}, pages 143--172.
  Amer. Math. Soc., Providence, RI, 2003.

\bibitem{GHL:TAMS2003}
Alexander Grigor'yan, Jiaxin Hu, and Ka-Sing Lau.
\newblock Heat kernels on metric measure spaces and an application to
  semilinear elliptic equations.
\newblock {\em Trans. Amer. Math. Soc.}, 355(5):2065--2095, 2003.

\bibitem{GrigLiu15}
Alexander Grigor'yan and Liguang Liu.
\newblock Heat kernel and {L}ipschitz-{B}esov spaces.
\newblock {\em Forum Math.}, 27(6):3567--3613, 2015.

\bibitem{HSC01}
W.~Hebisch and L.~Saloff-Coste.
\newblock On the relation between elliptic and parabolic {H}arnack
  inequalities.
\newblock {\em Ann. Inst. Fourier (Grenoble)}, 51(5):1437--1481, 2001.

\bibitem{Hein-Lect}
Juha Heinonen.
\newblock {\em Lectures on analysis on metric spaces}.
\newblock Universitext. Springer Verlag New York, 2001.

\bibitem{nages_book}
Juha Heinonen, Pekka Koskela, Nageswari Shanmugalingam, and Jeremy~T. Tyson.
\newblock {\em Sobolev spaces on metric measure spaces: An approach based on
  upper gradients}, volume~27 of {\em New Mathematical Monographs}.
\newblock Cambridge University Press, Cambridge, 2015.

\bibitem{HinoS2}
Masanori Hino.
\newblock On singularity of energy measures on self-similar sets.
\newblock {\em Probab. Theory Related Fields}, 132(2):265--290, 2005.

\bibitem{HinoS1}
Masanori Hino and Kenji Nakahara.
\newblock On singularity of energy measures on self-similar sets. {II}.
\newblock {\em Bull. London Math. Soc.}, 38(6):1019--1032, 2006.

\bibitem{HT-Hodge}
Michael Hinz and Alexander Teplyaev.
\newblock Local {D}irichlet forms, {H}odge theory, and the {N}avier-{S}tokes
  equations on topologically one-dimensional fractals.
\newblock {\em Trans. Amer. Math. Soc.}, 367(2):1347--1380, 2015.

\bibitem{HT-curl}
Michael Hinz and Alexander Teplyaev.
\newblock Densely defined non-closable curl on carpet-like metric measure
  spaces.
\newblock {\em Math. Nachr.}, 291(11-12):1743--1756, 2018.

\bibitem{HZ}
Jiaxin Hu and Martina Z\"ahle.
\newblock Potential spaces on fractals.
\newblock {\em Studia Math.}, 170(3):259--281, 2005.

\bibitem{hua2017liouville}
Bobo Hua.
\newblock Liouville theorem for bounded harmonic functions on manifolds and
  graphs satisfying non-negative curvature dimension condition.
\newblock {\em Calc. Var. Partial Differential Equations}, 58(2):Paper No. 42,
  8, 2019.

\bibitem{IRS13}
Marius Ionescu, Luke~G. Rogers, and Robert~S. Strichartz.
\newblock Pseudo-differential operators on fractals and other metric measure
  spaces.
\newblock {\em Rev. Mat. Iberoam.}, 29(4):1159--1190, 2013.

\bibitem{ITR12}
Marius Ionescu, Luke~G. Rogers, and Alexander Teplyaev.
\newblock Derivations and {D}irichlet forms on fractals.
\newblock {\em J. Funct. Anal.}, 263(8):2141--2169, 2012.

\bibitem{Kajino}
Naotaka Kajino.
\newblock Spectral asymptotics for {L}aplacians on self-similar sets.
\newblock {\em J. Funct. Anal.}, 258(4):1310--1360, 2010.

\bibitem{KSW12}
Daniel~J. Kelleher, Benjamin~A. Steinhurst, and Chuen-Ming~M. Wong.
\newblock From self-similar structures to self-similar groups.
\newblock {\em Internat. J. Algebra Comput.}, 22(7):1250056, 16, 2012.

\bibitem{KigamiDendrites}
Jun Kigami.
\newblock Harmonic calculus on limits of networks and its application to
  dendrites.
\newblock {\em J. Funct. Anal.}, 128(1):48--86, 1995.

\bibitem{KigB}
Jun Kigami.
\newblock {\em Analysis on fractals}, volume 143 of {\em Cambridge Tracts in
  Mathematics}.
\newblock Cambridge University Press, Cambridge, 2001.

\bibitem{Kig:RFQS}
Jun Kigami.
\newblock Resistance forms, quasisymmetric maps and heat kernel estimates.
\newblock {\em Mem. Amer. Math. Soc.}, 216(1015):vi+132, 2012.

\bibitem{KS93}
Nicholas~J. Korevaar and Richard~M. Schoen.
\newblock Sobolev spaces and harmonic maps for metric space targets.
\newblock {\em Comm. Anal. Geom.}, 1(3-4):561--659, 1993.

\bibitem{KST04}
Pekka Koskela, Nageswari Shanmugalingam, and Jeremy~T. Tyson.
\newblock Dirichlet forms, {P}oincar\'{e} inequalities, and the {S}obolev
  spaces of {K}orevaar and {S}choen.
\newblock {\em Potential Anal.}, 21(3):241--262, 2004.

\bibitem{Kus89}
Shigeo Kusuoka.
\newblock Dirichlet forms on fractals and products of random matrices.
\newblock {\em Publ. Res. Inst. Math. Sci.}, 25(4):659--680, 1989.

\bibitem{LP06}
M.~L. Lapidus and E.~P.~J. Pearse.
\newblock A tube formula for the {K}och snowflake curve, with applications to
  complex dimensions.
\newblock {\em J. London Math. Soc. (2)}, 74(2):397--414, 2006.

\bibitem{LPW11}
M.~L. Lapidus, E.~P.~J. Pearse, and S.~Winter.
\newblock Pointwise tube formulas for fractal sprays and self-similar tilings
  with arbitrary generators.
\newblock {\em Adv. Math.}, 227(4):1349--1398, 2011.

\bibitem{Ledoux}
Michel Ledoux.
\newblock Isoperimetry and {G}aussian analysis.
\newblock In {\em Lectures on probability theory and statistics
  ({S}aint-{F}lour, 1994)}, volume 1648 of {\em Lecture Notes in Math.}, pages
  165--294. Springer, Berlin, 1996.

\bibitem{Lind}
Tom Lindstr{\o}m.
\newblock Brownian motion on nested fractals.
\newblock {\em Mem. Amer. Math. Soc.}, 83(420):iv+128, 1990.

\bibitem{MMS}
N.~Marola, M.~Miranda, Jr., and N.~Shanmugalingam.
\newblock Characterizations of sets of finite perimeter using heat kernels in
  metric spaces.
\newblock {\em Potential Anal.}, 45(4):609--633, 2016.

\bibitem{Mattila}
P.~Mattila.
\newblock {\em Geometry of sets and measures in {E}uclidean spaces}, volume~44
  of {\em Cambridge Studies in Advanced Mathematics}.
\newblock Cambridge University Press, Cambridge, 1995.
\newblock Fractals and rectifiability.

\bibitem{Ne05}
V.~Nekrashevych.
\newblock {\em Self-similar groups}, volume 117 of {\em Mathematical Surveys
  and Monographs}.
\newblock American Mathematical Society, Providence, RI, 2005.

\bibitem{NT08}
V.~Nekrashevych and A.~Teplyaev.
\newblock Groups and analysis on fractals.
\newblock In {\em Analysis on graphs and its applications}, volume~77 of {\em
  Proc. Sympos. Pure Math.}, pages 143--180. Amer. Math. Soc., Providence, RI,
  2008.

\bibitem{P-P10}
K.~Pietruska-Pa{\l}uba.
\newblock Heat kernel characterisation of {B}esov-{L}ipschitz spaces on metric
  measure spaces.
\newblock {\em Manuscripta Math.}, 131(1-2):199--214, 2010.

\bibitem{PW14}
D.~Pokorn\'y and S.~Winter.
\newblock Scaling exponents of curvature measures.
\newblock {\em J. Fractal Geom.}, 1(2):177--219, 2014.

\bibitem{RZ12}
J.~Rataj and M.~Z\"ahle.
\newblock Curvature densities of self-similar sets.
\newblock {\em Indiana Univ. Math. J.}, 61(4):1425--1449, 2012.

\bibitem{Sab99}
C.~Sabot.
\newblock Pure point spectrum for the {L}aplacian on unbounded nested fractals.
\newblock {\em J. Funct. Anal.}, 173(2):497--524, 2000.

\bibitem{Sion64}
M.~Sion.
\newblock A characterization of weak{$^{\ast} $} convergence.
\newblock {\em Pacific J. Math.}, 14:1059--1067, 1964.

\bibitem{St99}
R.~S. Strichartz.
\newblock Some properties of {L}aplacians on fractals.
\newblock {\em J. Funct. Anal.}, 164(2):181--208, 1999.

\bibitem{Str}
R.~S. Strichartz.
\newblock Fractafolds based on the {S}ierpi\'nski gasket and their spectra.
\newblock {\em Trans. Amer. Math. Soc.}, 355(10):4019--4043, 2003.

\bibitem{St05p}
R.~S. Strichartz.
\newblock Analysis on products of fractals.
\newblock {\em Trans. Amer. Math. Soc.}, 357(2):571--615, 2005.

\bibitem{St09p}
R.~S. Strichartz.
\newblock A fractal quantum mechanical model with {C}oulomb potential.
\newblock {\em Commun. Pure Appl. Anal.}, 8(2):743--755, 2009.

\bibitem{ST}
R.~S. Strichartz and A.~Teplyaev.
\newblock Spectral analysis on infinite {S}ierpi\'nski fractafolds.
\newblock {\em J. Anal. Math.}, 116:255--297, 2012.

\bibitem{T08}
A.~Teplyaev.
\newblock Harmonic coordinates on fractals with finitely ramified cell
  structure.
\newblock {\em Canad. J. Math.}, 60(2):457--480, 2008.

\bibitem{W15}
S.~Winter.
\newblock Minkowski content and fractal curvatures of self-similar tilings and
  generator formulas for self-similar sets.
\newblock {\em Adv. Math.}, 274:285--322, 2015.

\bibitem{WZ13}
S.~Winter and M.~Z\"ahle.
\newblock Fractal curvature measures of self-similar sets.
\newblock {\em Adv. Geom.}, 13(2):229--244, 2013.

\end{thebibliography}

\end{document}